\documentclass[11pt,onesided]{amsart}
\usepackage[all]{xypic}  	
\usepackage{graphicx}				
\usepackage{amssymb,mathrsfs}
\usepackage{amssymb,amsthm,amsmath}
\usepackage{tikz}
\usepackage{color}
\usepackage{enumerate}
\usepackage{accents,upgreek,enumerate}
\usepackage[headings]{fullpage}
\usepackage{bm}
\usetikzlibrary{positioning}
\usetikzlibrary{matrix}
\usetikzlibrary{patterns}

\newtheorem{innercustomthm}{Theorem}
\newenvironment{customthm}[1]
  {\renewcommand\theinnercustomthm{#1}\innercustomthm}
  {\endinnercustomthm}

\tikzstyle{bsq}=[rectangle, draw, thick, minimum width=0.7cm, minimum height=0.7cm]
\tikzstyle{bver}=[rectangle, draw, thick, minimum width=1cm, minimum height=2cm]
\tikzstyle{bhor}=[rectangle, draw, thick, minimum width=2cm, minimum height=1cm]

\newcommand{\dashedrightarrow}[1][2pt]{%
  \settowidth{\@tempdima}{$\rightarrow$}\rightarrow
  \makebox[-\@tempdima]{\hskip-1.5ex\color{white}\rule[0.5ex]{#1}{1pt}}
  \phantom{\rightarrow}
}

\DeclareMathAlphabet{\mathpzc}{OT1}{pzc}{m}{it}
\usepackage[backref]{hyperref}
\hypersetup{
  colorlinks   = true,          
  urlcolor     = blue,          
  linkcolor    = blue,          
  citecolor   = violet             
}

\address{Department of Mathematics, University of Kentucky, 719 Patterson Office Tower, Lexington, KY 40506 USA.}
\email{dave.jensen@uky.edu}

\address{Department of Pure Mathematics and Mathematical Statistics, Wilberforce Road, University of Cambridge, Cambridge CB2 1TP UK.}
\email{dr508@cam.ac.uk}

\newtheorem{theorem}{Theorem}[section]
\newtheorem{corollary}[theorem]{Corollary}
\newtheorem{lemma}[theorem]{Lemma}
\newtheorem{proposition}[theorem]{Proposition}
\newtheorem{definition}[theorem]{Definition}

\newtheorem{assumption}[theorem]{Assumption}

\newtheorem{quasi-theorem}[theorem]{Quasi-Theorem}
\newtheorem{blank remark}[theorem]{}

\newcommand{\trop}{\operatorname{trop}}
\newcommand{\ddiv}{\operatorname{div}}

\newcommand{\Pic}{\operatorname{Pic}}

\newtheorem{rem1}[theorem]{Remark}
\newenvironment{remark}{\begin{rem1}\em}{\end{rem1}}

\newtheorem{not1}[theorem]{Notation}
\newenvironment{notation}{\begin{not1}\em}{\end{not1}}

\newcommand{\CC} {{\mathbb C}}

\newcommand{\NN} {{\mathbb N}}		
\newcommand{\PP}{\mathbb{P}}
		
\newcommand{\RR} {{\mathbb R}}		
\newcommand{\ZZ} {{\mathbb Z}}

\newcommand{\Hom}{\operatorname{Hom}}

\DeclareMathOperator{\spec}{Spec}

\newcommand{\cal}{\mathcal}

\def\cD{{\cal D}}

\def\cM{{\cal M}}

\def\cO{{\cal O}}

\def\cW{{\cal W}}

\newcommand{\Mbar}{\overline{\cM}}

\def\trop{\mathrm{trop}}
\def\an{\mathrm{an}}

\makeatletter
\def\blfootnote{\xdef\@thefnmark{}\@footnotetext}
\makeatother

\title{Brill-Noether theory for curves of a fixed gonality}

\author[David Jensen {\it \&} Dhruv Ranganathan]{\vspace{-0.0in}{{\larger D}{\smaller avid}\ \ {\larger J}{\smaller ensen} {\it \&} {\larger D}{\smaller hruv}\ \ {\larger R}{\smaller anganathan}}}
\date{}

\begin{document}

\begin{abstract}
We prove a generalization of the Brill--Noether theorem for the variety of special divisors $W^r_d(C)$ on a general curve $C$ of prescribed gonality. Our main theorem gives a closed formula for the dimension of $W^r_d(C)$. We build on previous work of Pflueger, who used an analysis of the tropical divisor theory of special chains of cycles to give upper bounds on the dimensions of Brill--Noether varieties on such curves. We prove his conjecture, that this upper bound is achieved for a general curve. Our methods introduce logarithmic stable maps as a systematic tool in Brill--Noether theory. A precise relation between the divisor theory on chains of cycles and the corresponding tropical maps theory is exploited to prove new regeneration theorems for linear series with negative Brill--Noether number. The strategy involves blending an analysis of obstruction theories for logarithmic stable maps with the geometry of Berkovich curves. To show the utility of these methods, we provide a short new derivation of lifting for special divisors on a chain of cycles with generic edge lengths, proved using different techniques by Cartwright, Jensen, and Payne. A crucial technical result is a new realizability theorem for tropical stable maps in obstructed geometries, generalizing a well-known theorem of Speyer on genus $1$ curves to arbitrary genus.
\end{abstract}

\maketitle

\vspace{-0.2in}

\section{Introduction}
\label{Sec:Intro}

Given a smooth projective curve $C$ over the complex numbers, let $W^r_d (C)$ denote the subvariety of the Picard variety of $C$, parameterizing divisors of degree $d$ that move in a linear system of dimension at least $r$.  The dimensions of these varieties are fundamental invariants of $C$. For specific curves, the dimension of $W^r_d(C)$ can be far from that predicted by a naive dimension count.  However, the Brill--Noether theorem asserts that, if $C$ is a general curve of genus $g$, this naive dimension count is correct. More precisely, the dimension of $W^r_d (C)$ is
\[
\rho (g,r,d) := g-(r+1)(g-d+r),
\]
where a scheme is understood to be empty when its dimension is negative. This result was first proved in a seminal paper by Griffiths and Harris~\cite{GH80}. It has inspired a great deal of additional research, including multiple re-proofs by an array of different techniques~\cite{tropicalBN,EH86,Laz86}.

The main result of this paper is an analogue of this theorem for curves of a prescribed gonality.  Fix an integer $k \geq 2$, and let $C$ be a general curve of genus $g$ and gonality $k$.  In \cite{Pfl16b}, Pflueger defines the invariant
\[
{\rho}_k (g,r,d) : = \max_{\ell \in \{ 0, \ldots , r' \}} \rho (g,r-\ell,d) - \ell k ,
\]
where $r' = \min \{ r,g-d+r-1 \}$.
He shows that $\dim W^r_d (C) \leq {\rho}_k (g,r,d)$, and conjectures that equality holds.  In this paper, we prove Pflueger's conjecture. 

\begin{customthm}{A}
\label{thm:mainthm}
Let $C$ be a general curve of genus $g$ and gonality $k \geq 2$ over the complex numbers. Assume that the quantity $g-d+r$ is positive.  Then
\[
\dim W^r_d (C) = {\rho}_k (g,r,d) .
\]
In particular, $W^r_d (C)$ is non-empty precisely when ${\rho}_k (g,r,d)$ is non-negative.
\end{customthm}

When $C$ is general in $\cM_g$, the scheme $W^r_d(C)$ is equidimensional, but for a general $k$-gonal curve, it can have irreducible components of different dimensions.  For example, if $C$ is a general trigonal curve of genus 6, then $W^1_4 (C)$ has a 1-dimensional component whose general member is the $g^1_3$ plus a base point.  At the same time, however, such a curve can be embedded in $\PP^1 \times \PP^1$ as curve of bidegree $(3,4)$, and projection onto the second $\PP^1$ yields a $g^1_4$ that is base point free.  It is a straightforward exercise to see that this $g^1_4$ is an isolated point of $W^1_4 (C)$.  Theorem~\ref{thm:mainthm} should be interpreted as saying that $W^r_d (C)$ has a component of maximal possible dimension ${\rho}_k (g,r,d)$.

Our results rely on a newly observed connection between the divisor theory on tropical curves and the theory of logarithmic stable maps. A crucial technical result, which we believe to be of independent interest, is the following.

\begin{customthm}{B}\label{thm: lifting}
Let $\Gamma$ be a trivalent marked tropical curve with all vertex weights equal to $0$. Assume that the image of $[\Gamma]$ in $\cM_g^{\trop}$ is a chain of $g$ cycles. Let
\[
\Psi: \Gamma\to \RR^r
\]
be a balanced tropical map. Assume that each cycle of $\Gamma$ spans at least a hyperplane in $\RR^r$ and that any pair of consecutive cycles spans $\RR^r$. If $\Psi$ is naively well-spaced, then there exists a smooth curve $C$ over a non-archimedean extension of $\CC$ and a map $F: C \dashrightarrow   \mathbb G_m^r$ such that $F^\trop = \Psi$.
\end{customthm}

\noindent
\textit{Naive well-spacedness} is an explicit collection of piecewise linear conditions on the edge lengths of $\Gamma$ and of edges in the trees attached to the cycles. There is one such condition for each cycle of $\Gamma$ that fails to span $\RR^r$. See Definition~\ref{def: naive-well-spacedness}.

\subsection{Context and motivation}  Theorem~\ref{thm:mainthm} reduces to a number of known results in special cases.  The maximum possible gonality of a curve of genus $g$ is $k = \lfloor \frac{g+3}{2} \rfloor$. This is the gonality of a general curve of genus $g$ and in this case Theorem~\ref{thm:mainthm} recovers the Brill-Noether theorem.  At the other extreme, the minimum possible gonality of a positive genus curve is $k=2$, in which case the curve is hyperelliptic.  Here, Theorem~\ref{thm:mainthm} shows that $\dim W^r_d (C) = d-2r$.  This follows from the fact that every $g^r_d$ on a hyperelliptic curve can be obtained by adding $d-2r$ base points to $r\cdot g^1_2$. This dates back to at least the $19^{\textnormal{th}}$ century, and was likely known to Clifford~\cite{Cliff78}.

For intermediate values of $k$, progress has been much more recent.  Much like in the hyperelliptic case, when $d \geq kr$, one can always construct a $g^r_d$ by adding $d-kr$ base points to $r\cdot g^1_k$.  Applying Riemann-Roch, this shows that
\begin{eqnarray*}
\dim W^r_d (C) &\geq& \max \{ d-kr, (2g-2-d)-k(g-d+r-1) \} \\
&=& \rho (g,r-r',d) - r'k ,
\end{eqnarray*}
where, as above, $r' = \min \{ r,g-d+r-1 \}$.
In~\cite{Ballico96}, Ballico and Keem consider the case where $g \leq 4k-4$.  In this case, they show that $\dim W^r_d(C)$ can exceed $\rho (g,r,d)$ by at most $g-2k+2$, using a blend of admissible cover techniques and limit linear series. In \cite{CM99}, Coppens and Martens exhibit components of $W^r_d (C)$ of dimension $\rho (g,r-\ell,d) - \ell k$ for $\ell$ equal to 0, 1, and $r'$.  In \cite{CM02}, they expand this to the case where $r'+1-\ell$ divides either $r'$ or $r'+1$ and is smaller thank $k$.  Together with Pflueger's upper bound, these results determine the dimension of $W^r_d (C)$ for general trigonal or tetragonal curves.

In a related but orthogonal direction, Farkas considers the Brill--Noether locus in the moduli space of $k$-gonal curves~\cite{Farkas}. When $\rho (g,r,d)$ is negative, he asks whether this locus has the expected codimension $-\rho (g,r,d)$.  He shows that this is the case when $k \geq 2 + r(2-\rho(g,r,d))$.  Farkas's approach uses the geometry of the Kontsevich moduli space of stable maps in an essential fashion.  To our knowledge, this is the only result on the codimension of the Brill-Noether locus in the moduli space of $k$-gonal curves.  In $\mathcal{M}_g$, more is known, and in particular there are many results when $-\rho (g,r,d)$ is relatively small \cite{EH86, Edidin93, PfluegerThesis}.  When $r=1$, the Brill-Noether locus is the image of the Hurwitz space, and is therefore irreducible of the expected dimension.  However, counterexamples appear in rank $2$. For example, by Clifford's Theorem a curve has a $g^2_4$ if and only if it is hyperelliptic, so the codimension of the space of curves with a $g^2_4$ is $g-2$. On the other hand, the expected codimension is $- \rho (g,2,4) = 2g-6$.

Theorem~\ref{thm:mainthm} provides many cases where the Brill-Noether number is negative but the general curve of gonality $k$ nevertheless possesses a $g^r_d$.  In these cases, the Brill-Noether locus necessarily fails to have the expected codimension.  The invariants for which $W^r_d(C)$ is non-empty for a general curve of gonality $k$ are precisely determined by Theorem~\ref{thm:mainthm}. Following \cite{Pfl16b}, we represent this graphically by identifying a choice of $d$ and $r$ with the point $(r+1,g-d+r)$ in the plane, and identifying the region where $W^r_d (C)$ is nonempty with the area bounded by a curve in the first quadrant. We refer to this region as the \textit{BN region}. Figure~\ref{Fig:Empty} depicts the situation for different values of $k$, ranging from most general to most special.  By the Brill--Noether theorem, when $k = \lfloor \frac{g+3}{2} \rfloor$, the BN region lies below the hyperbola $xy = g$.  At the other extreme, when $k=2$ the BN region lies below the line $x+y = g+1$.  For more general values of $k$, we see by Theorem~\ref{thm:mainthm} that the BN region region lies below the graph of the piecewise quadratic equation
\[
g = \min_{\ell \in \{ 0, \ldots , r' \}} (x-\ell)(y-\ell) + \ell k .
\]
The expression $(x-\ell)(y-\ell) + \ell k$ obtains its minimum at $\ell = \frac{1}{2} (x+y-k)$.  When this number is nonnegative and smaller than $r'$, we see that the formula above is quadratic in $x$ and $y$.  This corresponds to the curved region in the center of Figure~\ref{Fig:Empty}.  Otherwise, the minimum in the expression above is obtained at $\ell = 0$ or $\ell = r'$, corresponding to the two linear regions at either end of Figure~\ref{Fig:Empty}.

\begin{figure}[h!]
\begin{tikzpicture}

\foreach \r in {0.5,1,...,8}
	\foreach \s in {0.5,1,...,8}
	\draw [ball color=black] (\r,\s) circle (0.125mm);

\draw[->] (0,0)--(0,8) node[above] {$g-d+r$};
\draw[->] (0,0)--(8,0) node[right] {$r+1$};
\draw[color=black, thick, smooth, domain=0.125:4] plot ({\x},{1/(\x)});
\draw[color=black, smooth, thick,domain=4:8] plot ({\x},{1/(\x)});
\draw[thick,domain=0.125:8] plot ({\x},{8.125-\x});
\draw[thick,domain=1:3] plot ({\x-0.14},{\x+1.8575-2*sqrt(2*\x-2)});
\draw[thick,] (0.125,8)--(0.862,2.85);
\draw[thick,color=black] (2.8,0.86275)--(8,0.125);
\draw (1,0.25) node {\footnotesize $k=\lfloor \frac{g+3}{2} \rfloor$};
\draw (4.5,4.5) node {\footnotesize $k=2$};
\draw (2.75,1.75) node {\footnotesize $2 < k < \lfloor \frac{g+3}{2} \rfloor$};
\draw [ball color=black] (0,0) circle (0.5mm);

\end{tikzpicture}
\caption{This graph shows the points $(r+1,g-d+r)$ for which $W^r_d (C)$ is nonempty, when $C$ is a general curve of gonality $k$, for various values of $k$.}
\label{Fig:Empty}
\end{figure}


\subsection{Approach and techniques} The upper bound
\[
\dim W^r_d (C) \leq {\rho}_k (g,r,d)
\]
established by Pflueger is based on tropical techniques, which in turn rely on results in Berkovich geometry. He analyzes the divisor theory of a particular metric graph $\Gamma$ of gonality $k$. This graph is a chain of cycles with special edge lengths, chosen in such a way that $\Gamma$ supports a degree $k$ divisor of rank $1$. It is a generalization of the generic chain of cycles used in~\cite{tropicalBN} to establish a new proof of the Brill--Noether theorem. In addition to this combinatorial analysis, Pflueger uses the lifting results of~\cite{ABBR} to show that there is a curve $C$ of gonality $k$ that specializes to $\Gamma$. He first uses the gonality pencil to construct a map $\Gamma\to \RR$ and then lifts this map to a family of algebraic covers of $\PP^1$. The lifting result for such maps can be seen via the logarithmic unobstructedness of the space of admissible covers of $\PP^1$. The upper bound follows from Baker's specialization lemma~\cite{Bak08}.

To complete Pflueger's argument, we must show that, in addition to the gonality pencil, certain additional divisors on $\Gamma$ of higher rank lift to $C$ without dropping rank.  While there is a fairly good understanding of lifting divisors of rank $1$ on tropical curves~\cite{ABBR,ABBR2,LM14}, there are only a small number of lifting results in higher rank. Indeed, Cartwright has shown that the lifting problem for rank $2$ divisors on tropical curves satisfies a version of Mn\"ev universality~\cite{Cart15}. See~\cite[Section 10]{BJ16} for a detailed discussion of the lifting problem in the tropical setting.

One crucial family of graphs for which a complete higher rank lifting result for divisors is known is the \textit{generic} chain of cycles~\cite{CJP}. The argument used to establish this result in~\cite{CJP} relies heavily on the known structure of $W^r_d (C)$ for general curves $C$, including its determinantal description, its dimension, and its class as a cycle in the Jacobian.  These techniques do not appear to generalize to special chains of cycles. Instead, we appeal to logarithmic geometry to prove the necessary lifting result. Along the way, we provide a short new derivation of the analogous lifting result for the generic chain of cycles, see Theorem~\ref{Thm:GenericLift}. We use this to give a purely tropical proof of the existence part of the Brill--Noether theorem~\cite{Kempf,KleimanLaksov} that does not rely on intersection theory; see Theorem~\ref{thm: Kempf-Kleiman-Laksov}.

When $\Gamma$ has generic edge lengths, our lifting argument proceeds as follows. Given a special divisor $D$ on $\Gamma$, distinguished piecewise linear functions in the associated tropical linear series can be used to construct a tropical stable map from $\Gamma$ to $\PP^r_{\trop}$. The slice by any coordinate hyperplane is equivalent to $D$ on $\Gamma$.  We then use deformation theory of logarithmic stable maps to lift the tropical map, recovering the divisor as the tropicalization of a hyperplane section. A key observation in this argument is that the genericity of $\Gamma$ guarantees that the map to $\PP^r_{\trop}$ is \textit{non-superabundant} in the sense of~\cite{CFPU,KatLift,Mi03}.  That is, the image of each cycle of $\Gamma$ spans the target $\RR^r$. This combinatorial condition guarantees the smoothness of a versal deformation space naturally associated to this tropical stable map, so there is no obstruction to lifting.

Given a divisor on a $k$-gonal chain of cycles $\Gamma$ that has rank $r$, one may still construct a map $\Gamma\to \PP^r_{\trop}$. However, the image of each cycle spans a linear space of dimension at most $k-1$, and many consecutive cycles may lie in the same linear subspace. Such tropical curves are called \textit{superabundant}. The algebraic deformation space encoded by this tropical curve is highly singular. The general tropical realization problem for tropical stable maps also satisfies a version of Mn\"ev universality, due to Vakil's Law, see~\cite{Vak06} and~\cite[Theorem D {\it \&} Remark 3.1.1]{R16a}. The relevant tropical realization problem for maps $\Gamma\to\PP^r_{\trop}$ seems beyond the reach of present tropical realizability theorems.

Rather than considering maps to projective space, in the $k$-gonal setting it appears more natural to consider relative maps to projective space over the base $\PP^1$.  In other words, we consider maps from the curve to projective bundles over $\PP^1$ such that the composition with the map to $\PP^1$ yields the $k$-fold cover.  For simplicity, we refer to such projective bundles throughout as \textit{scrolls}.  Our work strongly suggests that, when $C$ is a general $k$-gonal curve, the various components of $W^r_d (C)$ correspond to different scrolls.  In \cite[Question~1.10]{Pfl16b}, Pflueger asks whether every component of $W^r_d (C)$ has dimension $\rho (g,r-\ell,d)-\ell k$ for some $\ell$.  A consequence of Proposition~\ref{Prop:Remainder} and the results of Section~\ref{Sec:SpecialLift} is that, for each value of $\ell$ in the range $r'+1-k < \ell \leq r'+1$, there exists a scroll $\mathbb{S}(a,b)$ such that the corresponding space of nondegenerate maps from $C$ to $\mathbb{S}(a,b)$ has dimension at least $\rho (g,r-\ell,d) - \ell k$.  Indeed, this conjectural description has been confirmed in the years since this paper first appeared.  The Brill-Noether variety $W^r_d (C)$ admits a stratification by isomorphism classes of scrolls, and each stratum has a certain expected dimension \cite{HLarson,CPJ}.

Studying maps to scrolls allows us to replace the tropical realizability problem above with a simpler one.  Rather than constructing a map from $\Gamma$ to a tropical projective space, we construct a map to the tropicalization of a scroll. The corresponding tropical curve is still superabundant, but has fewer obstructions. A crucial technical achievement of this paper is the solution to the relevant tropical realizability problem (Theorem~\ref{thm: lifting}). The necessary lifting theorem for divisors is obtained as a corollary. It should be noted that the chain of cycles is among a small number of tropical curves whose Brill--Noether theory can be explicitly understood and it exhibits a number of unexpected properties. One among these properties, observed and exploited in this text, is a close relationship between the theory of divisors and that of maps to projective space, which typically breaks down in the tropical setting.

The realizability problem for tropical stable maps remains largely unsolved, despite substantial interest~\cite{ABBR,BPR16,CFPU,KatLift,Mi03,Ni15,NS06,R16a,R15a,RSW19,Sp07}. Two important cases where progress has been made are Speyer's genus $1$ results~\cite{RSW19,Sp07} and the results of  Cheung, Fantini, Park, and Ulirsch for non-superabundant curves~\cite{CFPU}. The latter builds on insights from Nishinou and Siebert's pioneering genus $0$ analysis~\cite{NS06}. To the best of our knowledge, Theorem~\ref{thm: lifting} provides the first solution to a superabundant tropical realizability problem that applies in all genera and all target dimensions, to a maximally degenerate combinatorial type. Our approach blends the techniques of the aforementioned authors. We use the geometry of logarithmic stable maps and lifting theorems for tropical intersections to reduce the question to one about maps to $\PP^1$. We then use semistable vertex decompositions for morphisms of curves~\cite{ABBR,BPR16} to further reduce to a local lifting question in genus $1$. This allows for Speyer's results to be applied ``cycle-by-cycle''. We conclude by applying the tropicalization of moduli spaces of stable maps in superabundant geometries developed in~\cite{R16a}. For clarity, we focus on the case of chains of cycles in this text, but the procedural aspects of our proof should generalize to new combinatorial geometries and give high genus realizability theorems. On the other hand, we expect that improved lifting theorems in the chain of cycles geometry will have additional applications in the theory of linear series.

We conclude the introduction by noting that the only substantial difficulty in extending these results to positive characteristic lies in generalizing Theorem~\ref{thm: lifting}. In particular, we rely on Speyer's genus $1$ lifting result~\cite[Theorem 3.4]{Sp07}, which requires a characteristic $0$ hypothesis.

\subsection*{Acknowledgments}  We are indebted to Sam Payne, who collaborated actively with us during the early stages of this project, and to Nathan Pflueger for discussions on his work. We thank Dan Abramovich and Davesh Maulik for patiently fielding our questions about obstruction theories, and Dori Bejleri for many helpful conversations. Dan Abramovich, Sam Payne, and an anonymous referee provided valuable feedback on drafts of the manuscript. Work on this project was begun when the authors were at Yale University together in Spring 2016 and the Fields Institute in December 2016. The work was completed when the second author was a member of the mathematics department at MIT. We thank these institutions for a warm and friendly atmosphere. The first author is supported by NSF DMS-1601896.

\setcounter{tocdepth}{1}
\tableofcontents

\section{Combinatorics of special chains of cycles}
\label{Sec:TheGraph}

In this section, we review the divisor theory of special chains of cycles, as discussed in~\cite{Pfl16b,Pfl16a}. The material of this section is developed in detail in those papers and their precursors \cite{tropicalBN, JP14}.  For more on the divisor theory of metric graphs, we refer the reader to \cite{Bak08, BJ16}.

\subsection{Torsion profiles and displacement tableaux} Let $\Gamma$ be a chain of cycles with bridges, as pictured in Figure \ref{Fig:TheGraph}.  Note that $\Gamma$ has $g$ cycles, labeled $\gamma_1 , \ldots , \gamma_g$, and $2g$ vertices, one on the lefthand side of each cycle, which we label $v_1, \ldots , v_g$, and one on the righthand side of each cycle, which we label $w_1, \ldots, w_g$.  There are two edges connecting the vertices $v_j$ and $w_j$, the top and bottom edges of the $j^{\textnormal{th}}$ cycle $\gamma_j$, whose lengths are denoted $\ell_j$ and $m_j$, respectively.
\begin{figure}[h!]
\begin{tikzpicture}

\draw (0,0) circle (1);
\draw (-1.25,0) node {\footnotesize $v_1$};
\draw (0.75,0) node {\footnotesize $w_1$};
\draw (1,0)--(2,0);
\draw (3,0) circle (1);
\draw (2.25,0) node {\footnotesize $v_2$};
\draw (4,0)--(5,0);
\draw (6,0) circle (1);
\draw (7,0)--(8,0);
\draw (9,0) circle (1);
\draw (9.6,0) node {\footnotesize $w_{g-1}$};
\draw (10,0)--(11,0);
\draw (12,0) circle (1);
\draw (11.25,0) node {\footnotesize $v_g$};
\draw (13.3,0) node {\footnotesize $w_g$};

\draw [<->] (7.15,0.25)--(7.85,0.25);
\draw [<->] (7.15,0.5) arc[radius = 1.15, start angle=10, end angle=170];
\draw [<->] (7.15,-0.5) arc[radius = 1.15, start angle=-9, end angle=-173];

\draw (7.5,0.4) node {\footnotesize$n_j$};
\draw (6,1.75) node {\footnotesize$\ell_j$};
\draw (6,-1.75) node {\footnotesize$m_j$};

\draw [ball color=black] (-1,0) circle (0.5mm);
\draw [ball color=black] (1,0) circle (0.5mm);

\draw [ball color=black] (2,0) circle (0.5mm);
\draw [ball color=black] (4,0) circle (0.5mm);

\draw [ball color=black] (5,0) circle (0.5mm);
\draw [ball color=black] (7,0) circle (0.5mm);

\draw [ball color=black] (8,0) circle (0.5mm);
\draw [ball color=black] (10,0) circle (0.5mm);

\draw [ball color=black] (11,0) circle (0.5mm);
\draw [ball color=black] (13,0) circle (0.5mm);
\end{tikzpicture}
\caption{The chain of cycles $\Gamma$.}
\label{Fig:TheGraph}
\end{figure}
For $1 \leq j \leq g-1$ there is a bridge connecting $w_j$ and $v_{j+1}$, which we refer to as the $j^{\textnormal{th}}$ bridge $\beta_j$, of length $n_j$.  The failure of a chain of cycles to be Brill--Noether general is measured by the \textit{torsion orders} of the cycles~\cite{Pfl16a}.

\begin{definition}\label{Def:TorsionOrder}
If $\ell_j + m_j$ is an irrational multiple of $m_j$, then the $j^{\textnormal{th}}$ \textbf{torsion order} $\mu_j$ is 0.  Otherwise, $\mu_j$ is the minimum positive integer such that $\mu_j m_j$ is an integer multiple of $\ell_j + m_j$.  The sequence $\vec{\mu} = (\mu_1 , \mu_2 , \ldots , \mu_g )$ is called the \textbf{torsion profile} of the chain of cycles.
\end{definition}

The special divisor classes on $\Gamma$, i.e. the classes of degree $d$ and rank greater than $d-g$, are classified in \cite{Pfl16a}.  It is shown that the Brill--Noether locus $W^r_d (\Gamma)$ parametrizing divisor classes of degree $d$ and rank $r$ is a union of tori, indexed by certain types of tableaux, called \textit{$\vec{\mu}$-displacement tableaux}. These are generalizations of standard Young tableaux on the alphabet $\{1,\ldots, g\}$. The generalization allows the same letter to appear in multiple boxes in the tableaux, in a precise manner dictated by $\vec\mu$.

\begin{definition}
A tableau is a \textbf{$\vec{\mu}$-displacement tableau} if, for any two boxes containing the same symbol $j$, the lattice distance between them is a multiple of $\mu_j$.
\end{definition}

For consistency with \cite{CJP,tropicalBN, JP14,JP16A,JP16}, we write our tableaux in the English style, beginning in the upper left and proceeding down and to the right\footnote{In \cite{Pfl16b, Pfl16a}, Pflueger uses the French notation for tableaux, beginning in the bottom left.  Similarly, in \cite{tropicalBN} and subsequent papers, the tableau corresponding to a divisor of rank $r$ and degree $d$ has $r+1$ columns and $g-d+r$ rows, but in \cite{Pfl16b, Pfl16a} this convention is reversed.}.  The top row and leftmost column of a tableau are the $0^{\textnormal{th}}$ row and column, respectively.  We write $t(x,y)$ for the value appearing in column $x$ and row $y$ of the tableau $t$.

\subsection{Coordinates on $\Gamma$ and its Picard group} A chain of cycles with torsion profile $\vec \mu$ has a natural system of coordinates, obtained from the coordinates on each cycle~\cite{Pfl16b}. On the cycle $\gamma_j$, let $\langle \xi \rangle_j$ denote the point located $\xi m_j$ units from $w_j$ in the counterclockwise direction.  Note that the points $v_j$ and $w_j$ are equal to $\langle -1 \rangle_j$ and $\langle 0 \rangle_j$, respectively, and $\langle a_1 \rangle_j = \langle a_2 \rangle_j$ if and only if $a_1 \equiv a_2 \pmod{\mu_j}$.

The Jacobian of $\Gamma$ can be identified with the product of the cycles and has a system of coordinates induced from $\Gamma$, as do the higher degree Picard groups. In particular, any divisor of degree $d$ on $\Gamma$ is equivalent to a \textit{unique} divisor of the form
\[
(d-g)w_g + \sum_{j=1}^g \langle \xi_j \rangle_j
\]
for real numbers $\xi_j$.  By uniqueness, we may think of $\xi_j$ as a function on the Picard group $\mathrm{Pic}(\Gamma)$.  This function is not linear on divisor classes.  However, as noted in \cite[Remark~3.5]{Pfl16b}, the function
\[
\tilde{\xi_j} = \xi_j - (j-1)
\]
is linear.  This is an immediate consequence of the tropical Abel-Jacobi Theorem~\cite{AbelJacobi}.  In particular, $\tilde{\xi_j}$ is obtained by integrating against a harmonic 1-form supported on the cycle $\gamma_j$.  We will make frequent use of both $\xi_j$ and $\tilde{\xi}_j$ throughout the text.

For a $\vec{\mu}$-displacement tableau $t$, we define a subtorus of the $d^{\mathrm{th}}$ Picard group of $\Gamma$ as
\[
\mathbf{T}(t) = \left\{ D = (d-g)w_g + \sum_{j=1}^g \langle \xi_j \rangle_j \ \big{\vert} \ \xi_j \in \mathbb{R}, \xi_{t(x,y)} \equiv y-x \pmod{\mu_j} \right\} .
\]
If $j$ does not appear in the tableau $t$, then there is no restriction placed on $\xi_j$.  It follows that the dimension of $\mathbf{T}(t)$ is equal to the number of elements of the set $\{ 1, \ldots , g \}$ that do not appear in the tableau $t$.  Moreover, if the number $j$ appears in multiple boxes of the tableau $t$, then the value $\xi_j$ is nevertheless well-defined modulo $\mu_j$, by the definition of a $\vec{\mu}$-displacement tableau.  The main result of \cite{Pfl16b} is that
\[
W^r_d (\Gamma) = \bigcup \mathbf{T}(t),
\]
where the union is over all rectangular $\vec{\mu}$-displacement tableaux with $r+1$ columns and $g-d+r$ rows. This description assumes $g-d+r is positive$, since otherwise $W^r_d(\Gamma)$ coincides with $\mathrm{Pic}^d(\Gamma)$.

\subsection{Lingering lattice paths} A useful combinatorial tool to describe reduced divisors on the chain of cycles is the \textit{lingering lattice path}, used heavily throughout~\cite{tropicalBN}. Associated to a tableau $t$ is a sequence of integer vectors $p_0 , \ldots , p_g \in \ZZ^r$ defined as follows.  Let $e_0 , \ldots, e_{r-1}$ denote the standard basis vectors in $\ZZ^r$, and let $e_r = (-1,-1, \ldots , -1)$.  We define
\[
p_j = (r,r-1, \ldots , 1) + \sum_{i=1}^j \sum_{\substack{x=0 \\ \exists y \text{ s.t. } t(x,y)=i}}^r e_x .
\]

\begin{remark}
This definition of the lingering lattice path is easily seen to coincide with that of \cite[Definition 4.3]{tropicalBN} in the case where the torsion profile $\vec{\mu}$ is identically zero.  There, $p_0$ is defined to be the vector $(r,r-1, \ldots , 1)$, and $p_j$ is defined recursively via the formula
\[
p_j = p_{j-1} + \sum_{\substack{x=0 \\ \exists y \text{ s.t. } t(x,y)=j}}^r e_x .
\]
When $\vec{\mu}$ is identically zero, the condition that $j$ appears in column $x$ can be satisfied by at most one value of $x$ for each $j$.  For simplicity, we define $p_j (r) = 0$ for all $j$.  By the definition of a tableau, we have $p_j (i) > p_j (i+1)$ for all $i$ and $j$.
\end{remark}

The relation between lattice paths and divisors on $\Gamma$ is given by the following proposition.

\begin{proposition}
\label{Prop:TableauToPath}
Given a $\vec{\mu}$-displacement tableau $t$ as above, let $\eta_1 , \ldots , \eta_g$ be real numbers such that $\eta_{t(x,y)} \equiv p_{t(x,y)-1} (x) \pmod{\mu_j}$.
Then, for $0 \leq i \leq r$, the divisors
\[
D_i = iv_1 + (r-i)w_g + \sum_{\substack{j=1 \\ \nexists y \text{ s.t. } t(i,y)=j}}^g \langle \eta_j - p_{j-1} (i) \rangle_j
\]
are all equivalent, and are elements of $\mathbf{T}(t)$.
\end{proposition}

Observe that the sum above is over all $j$ that are not contained in the $i^{\mathrm{th}}$ column of the tableau $t$.  In other words, $D_i$ has a chip on the $j^{\mathrm{th}}$ cycle $\gamma_j$ if and only if $j$ does not appear in the $i^{\mathrm{th}}$ column of $t$.

\begin{proof}
By \cite[Remark~3.5]{Pfl16b}, the function $\tilde{\xi}_j$ is a linear function on divisors.  Let $C_{ij}$ denote the number of terms smaller than $j$ that appear in the $i^{\mathrm{th}}$ column of the tableau $t$.  We compute
\[
\tilde{\xi}_j (D_i) =  1-i-j + \eta_j - p_{j-1} (i) + C_{ij} \\
= 1-r-j + \eta_j + C_{rj}.
\]
Since the right hand side does not depend on $i$, we see that $D_i \sim D_j$ for all $i$ and $j$.  Now, if $t(x,y) = j$, we have
\[
\eta_j = p_{t(x,y)-1} (x) = r-x+y - C_{rj}.
\]
Plugging this in, we obtain
\[
\tilde{\xi}_j (D_i) + j-1 \equiv y-x \pmod{\mu_j} .
\]
So $D_i \in \mathbf{T}(t)$.
\end{proof}

\subsection{Vertex avoiding divisors and bases of rational functions} By definition, a divisor in $W^r_d (\Gamma)$ has rank at least $r$.  It will be useful to focus on a subset of $W^r_d (\Gamma)$ consisting of divisors that have rank exactly equal to $r$.

\begin{definition}
We say that the divisor class $D$ described in Proposition \ref{Prop:TableauToPath} is \textbf{vertex avoiding} if for all $i$ and $j$,
\[
\eta_j -p_{j-1} (i) \neq -1 \pmod{\mu_j}
\]
and if
\[
\eta_j - p_{j-1} (i) \equiv 0 \pmod{\mu_j},
\]
then $t(i,y) = j$ for some $y$.
\end{definition}

The rank of such divisors in $W^r_d(\Gamma)$ is established by the following lemma.

\begin{lemma}
\label{Lem:ExactRank}
Let $t$ be a $\vec{\mu}$-displacement tableau with $r+1$ columns, and $D \in \mathbf{T}(t)$ a vertex avoiding divisor.  Then $D$ has rank exactly $r$.
\end{lemma}

\begin{proof}
By \cite{Pfl16b}, the divisor $D$ has rank at least $r$.  By definition, the divisor $D_r - rv_1$ is effective, has at most one chip on each cycle, and no chips on any of the vertices $v_j$ or $w_j$.  It follows that $D_r - rv_1$ is $v_1$-reduced.  Since $\deg_{v_1} (D_r - rv_1) = 0$, we see that $D$ cannot have rank larger than $r$.
\end{proof}




Unlike the case where $\Gamma$ has generic edge lengths, not every component $\mathbf{T}(t)$ contains vertex avoiding divisors.  If $D$ is a divisor of rank $r$ on $\Gamma$ whose class is vertex avoiding, then $D_i$ is the unique effective divisor $D_i \sim D$ such that $\deg_{v_1}(D_i) = i$ and $\deg_{w_g}(D_i) = r-i$.  Throughout, we will write $D=D_r$, and $\psi_i$ for a piecewise linear function on $\Gamma$ such that $D + \ddiv (\psi_i) = D_i$.  Note that $\psi_i$ is uniquely determined up to an additive constant, and for $i<r$ the slope of $\psi_i$ along the bridge $\beta_j$ is $p_j(i)$.   The function $\psi_r$ is constant, which justifies our convention that $p_j(r) = 0$ for all $j$.

\section{Deformations of maps: logarithmic and tropical}
\label{Sec:Deform}

The proof of the main lifting theorems, as well as the analysis of generic chains of cycles in the following section, will use the geometry of logarithmic prestable maps to toric varieties~\cite{AC11,Che10,GS13} and their relationship to Berkovich geometry~\cite{R16a,R15a,U16}. We provide a rapid overview of the relevant ideas and refer the reader to these references, as well as the surveys~\cite{ACGHOSS,ACMUW} for details.

A remark about the valued fields over which we work is in order before proceeding. Toric varieties and the moduli spaces of maps from curves to toric varieties are canonically defined over the integers; we will pull them back to the complex numbers and pass to their analytifications with respect to the trivial valuation. We remind the reader that the field of definition of a point in the analytification of a variety over the trivially valued field $\CC$ will typically be a transcendental valued field extension of $\CC$. In particular, maps from curves that we will consider in the sequel will typically be defined over an extension of $\CC(\!(t)\!)$, and their stable models will be nonconstant families over a valuation ring.

\subsection{Toric varieties, Artin fans, and tropicalization} Let $T$ be a torus of rank $r$ with character and cocharacter lattices $M$ and $N$ respectively. A complete fan $\Delta$ in $N_\RR$ defines a proper toric variety $Z$ which will always be considered to carry the canonical logarithmic structure coming from the toric boundary. The \textit{Artin fan} of $Z$ is the zero-dimensional logarithmic Artin stack defined by $\mathscr A_Z := [Z/T]$. This Artin stack is logarithmically \'etale over $\spec(\CC)$ and the logarithmic structure on $Z$ is obtained via pulling back the structure on $\mathscr A_Z$ via the quotient map~\cite[Section 2]{AW}.

We equip the ground field $\CC$ with the trivial valuation. The analytic formal fiber $Z^{\beth}$, as defined by Thuillier~\cite{Thu07,U16}, coincides with the Berkovich analytification since $Z$ is complete. It admits a continuous and proper tropicalization map
\[
\trop: Z^{\beth}\to \overline \Delta,
\]
where $\overline \Delta$ is the compactified cone complex of the fan $\Delta$, see~\cite[Section 2.6]{ACMUW}. The work of Ulirsch~\cite{U16} gives a stack theoretic enhancement, by identifying $\mathscr A_Z^\beth$ with the complex $\overline \Delta$. Precisely, there is a commutative diagram of underlying topological spaces
\[
\xymatrix{ &  \mathscr A_Z^{\beth} \ar[dd]^{\mu_Z} \\
Z^\beth \ar[ur]^{\mathrm{Quotient}} \ar[dr]_{\trop} & \\
& \overline \Delta,}
\]
where $\mu_Z$ is an isomorphism. 

\subsection{Tropical curves, tropical maps, and moduli cones} A \textit{prestable tropical curve} is a $1$-dimensional polyhedral complex $\Gamma$ with real edge lengths and finitely many $2$-valent vertices. It will be convenient for us to consider $\Gamma$ as being non-compact, by deleting $1$-valent vertices from the corresponding compact graph. For readers familiar with tropical curves in other contexts, we remark that we will not need to keep track of the genus function at the vertex at any stage, as we work with maximally degenerate curves. A reader who prefers to add this datum may do so.

\begin{definition}
A \textbf{tropical prestable map to $\Delta$} is a continuous morphism $\Psi: \Gamma\to \Delta$, such that
\begin{enumerate}[(A)]
\item The map $\Psi$ is a morphism of polyhedral complexes, i.e. every polyhedron in $\Gamma$ maps to a cone in $\Delta$.
\item Upon restriction to any edge $e$ of $\Gamma$, $\Psi$ has integral slope in $\Delta$, taken with respect to a primitive integral vector in the direction of $\Psi(e)$.
\item The map $\Psi$ is \textit{balanced}, i.e. at every point of $\Gamma$, the sum of outgoing slopes of $\Psi$ is zero.
\end{enumerate}
\end{definition}

By forgetting the fan structure on $\Delta$ and taking the minimal polyhedral complex structure on $\Gamma$, we obtain a continuous balanced map $\Gamma\to N_\RR$ which is referred to as a \textit{parametrized tropical curve} in much of the literature. We will refer to this as the \textit{underlying map} of the prestable map $\Gamma\to \Delta$. Concretely, it is obtained by forgetting all bivalent vertices in $\Gamma$, replacing $\Delta$ with its support, and considering the natural map between them.

\begin{definition}
The \textbf{combinatorial type} of a tropical prestable map $\Psi:\Gamma\to\Delta$, denoted $[\Psi:\Gamma\to\Delta]$, is the following data obtained from $\Psi$:
\begin{enumerate}[(A)]
\item The underlying finite graph of $\Gamma$.
\item For each vertex $v\in \Gamma$, the cone $\sigma_v$ of $\Delta$ containing $\Psi(v)$.
\item For each edge $e\in \Gamma$, the cone $\sigma_e$ to which $e$ maps, and the slope of $\Psi(e)$ in $\sigma_e$.
\end{enumerate}
\end{definition}

By a well-known construction~\cite{GS13,R16a}, the collection of all tropical stable maps of a given identification of their combinatorial type with $\Theta = [\Gamma\to \Delta]$ is naturally parametrized by a polyhedral cone $T_\Theta$. We will refer to this as the \textit{moduli cone} of tropical curves of type $\Theta$. It may be considered as a deformation space of $\Psi: \Gamma\to \Delta$.

\subsection{Logarithmic prestable maps} Given a proper toric variety $Z$, work of~\cite{AC11,AW, GS13} produces moduli spaces $\mathscr L(Z)$ and $\mathscr L(\mathscr A_Z)$ of logarithmic prestable maps to $Z$ and $\mathscr A_Z$ respectively. These are moduli stacks over the category of schemes. Given a map from a test scheme $S\to \mathscr L(Z)$, by definition, one obtains a logarithmic structure $(S,\mathscr M_S)$ on the base scheme, a nodal, logarithmically smooth, $S$-flat curve $\mathscr C\to (S,\mathscr M_S)$ and an arrow
\[
\mathscr C\to Z
\]
in the category of logarithmic schemes. A map obtained as above is referred to as a \textit{minimal logarithmic prestable map}. The terminology is meant to reflect that $\mathscr M_S$ is the logarithmic structure on $S$ with the the minimal restrictions that one can place in order to ensure that $\mathscr C\to Z$ is a map in the logarithmic category. General, non-minimal, prestable logarithmic maps are obtained by pulling back a minimal family along a logarithmic morphism
\[
(S,\mathscr M_{S}')\to (S,\mathscr M_S).
\]
By this description, the moduli space $\mathscr L(Z)$ parametrizes minimal logarithmic prestable maps to $Z$. Replacing $Z$ by $\mathscr A_Z$ one similarly obtains a moduli problem for (minimal) logarithmic prestable maps to the Artin fan.

For each minimal logarithmic prestable map $[C\to Z]$ over a geometric point, we obtain a combinatorial type for a tropical prestable map. The graph $\Gamma$ is taken to be the marked dual graph of $C$. Given a vertex $v\in \Gamma$ corresponding to a component $C_v$, we take the cone $\sigma_v$ to be the cone corresponding to the stratum to which the generic point of $C_v$ maps. Each edge $e$ of $\Gamma$ corresponding to a node $q\in C$ determines a cone $\sigma_e$ and a slope in it. Similarly, each infinite edge is also given a well-defined slope. We will have no reason to describe this type explicitly, so we refrain from describing this association in detail, see~\cite{R16a,R15a}.

\subsection{Tropical realizability}\label{sec: realizability} The moduli space $\mathscr L(Z)$ is in general extremely singular. In contrast, we have the following result of Abramovich and Wise~\cite{AW}.

\begin{theorem}
The stack $\mathscr L(\mathscr A_Z)$ is a logarithmically smooth Artin stack in the smooth topology.
\end{theorem}

A concrete consequence of this is that given a point $p\in \mathscr L(\mathscr A_Z)$, the complete local ring at $p$ is a regular local ring tensored with $\CC[\![Q]\!]$, where $Q$ is the monoid defining a toric variety.  In fact, if $\Theta$ is the combinatorial type associated to $p$, then $Q$ is the dual monoid to the cone $T_\Theta$ described previously. In the maximally degenerate cases that we consider here, the complete local ring at $p$ will be isomorphic to $\CC[\![Q]\!]$ itself.

\begin{notation}
For the remainder of the paper, we will have no need to work with the global moduli space of logarithmic maps. Instead, we will fix a moduli point $[C_0\to Z]$, which will always be clear from context, and understand $\mathscr L(Z)$ to be the deformation space of $[C_0\to Z]$ obtained by only examining those logarithmic strata of the moduli space that contain this point. We utilize analogous notation for maps to $\mathscr A_Z$.
\end{notation}

Despite its pathological singularity properties, the moduli space $\mathscr L(Z)$ carries a naturally perfect obstruction theory. By an observation of Abramovich and Wise, this perfect obstruction theory is obtained from the relative obstruction theory for the natural map $\mathscr L(Z)\to\mathscr L(\mathscr A_Z)$. Below, we record the aspect of this geometry that is relevant for us.

Consider a one-parameter family of logarithmic prestable maps $\mathscr C\to Z$ over a valuation ring $\spec(R)$. By~\cite{R15a,U16}, the skeleton $\overline \Gamma$ associated to the degeneration $\mathscr C$ maps onto the skeleton $\overline \Delta$ of $Z$.  Passing to the subset $\Gamma$ of $\overline{\Gamma}$ that maps to $\Delta$, we obtain a tropical prestable map. Since the analytic formal fiber of $\mathscr L(Z)$ consists of equivalence classes of degenerations of maps over valuation rings, this furnishes a \textit{set theoretic} tropicalization map from the locus of the formal fiber where the logarithmic structure is trivial: $\mathscr L(Z)^\beth \to T_\Theta$. Note that the $\beth$-space notation typically refers to the full formal fiber rather than the intersection with the locus where the logarithmic structure is trivial, however in this paper we will always use the latter. In practice, this means that the generic fiber of the curve associated to a moduli point in $\mathscr L(Z)^\beth$ is always taken to be smooth. 

The following is a local version of~\cite[Theorem D]{R16a}. We assume that the logarithmic map is \textit{stable}, i.e. there are no semistable components that are contracted to a point. This is merely to fit within the framework of~\cite{R16a}; stability does not play a crucial role in what follows. 

\begin{theorem}
Let $[C_0\to Z]$ be a logarithmic stable map of combinatorial type $\Theta$. There is a continuous and proper tropicalization map
\[
\trop: \mathscr L(\mathscr A_Z)^\beth\to T_\Theta.
\]
In particular, pre-composition with $\mathscr L(Z)\to \mathscr L(\mathscr A_Z)$ yields a continuous and proper tropicalization map
\[
\trop:\mathscr L(Z)^\beth \to T_\Theta
\]
that coincides with the set theoretic tropicalization map defined above.
\end{theorem}

\begin{definition}
A tropical prestable map $\Gamma \to \Delta$ of type $\Theta$ is said to be \textbf{realizable} if the corresponding point of $T_\Theta$ lies in the image of $\mathscr L(Z)^\beth$ in the tropicalization map above.
\end{definition}

The crux is that analytically locally near highly degenerate prestable maps, one may think of $\mathscr L(Z)\to \mathscr L(\mathscr A_Z)$, up to smooth factors, as being the inclusion of a subvariety of a toric variety, and we may tropicalize that subvariety. In the next section, we consider a situation where $\mathscr L(Z)$ meets the strata of $\mathscr L(\mathscr A_Z)$ in a smooth substack of expected dimension. In Section~\ref{Sec:Def}, we consider a more complicated setting, where the map from $\mathscr L(Z)$ to $\mathscr L(\mathscr A_Z)$ is far from transverse in this sense, and it will be analyzed directly.

\section{Lifting divisors on a generic chain of cycles}
\label{Sec:GenericLift}

By the results of~\cite{tropicalBN,JP14}, the chain of cycles with generic edge lengths may be thought of as a tropical incarnation of a Brill--Noether--Petri general curve. Together with specialization and tropical independence, this establishes the classical Brill--Noether and Gieseker--Petri theorems. A key combinatorial ingredient in these papers is the classification of special divisors on the generic chain of cycles, due to Cools, Draisma, Payne, and Robeva.  Subsequent work in this area, including work on the maximal rank conjecture~\cite{JP16A,JP16}, relies on an appropriate lifting theorem for such special divisors~\cite{CJP}. In this section, we give a new proof of the following.

\begin{theorem}
\label{Thm:GenericLift}
Let $\Gamma$ be a chain of $g$ cycles with generic edge lengths.  That is, the torsion profile of $\Gamma$ is identically zero.  Let $D \in W^r_d (\Gamma)$ be a vertex avoiding divisor class.  Then there exists a curve $C$ of genus $g$ and a divisor $\cD \in W^r_d (C)$ such that the tropicalization of $C$ is $\Gamma$ and the tropicalization of $\cD$ is $D$.
\end{theorem}

The main result of~\cite[Theorem 1.1]{CJP} is that \emph{any} curve $C$ of genus $g$ whose tropicalization is $\Gamma$ possesses a divisor class $\cD \in W^r_d (C)$ whose tropicalization is $D$. This confirmed a conjecture of Cools, Draisma, Payne, and Robeva~\cite[Conjecture 1.5]{tropicalBN}. It may be possible to prove this stronger result using these techniques by more carefully analyzing the deformation theory, but we will only require the weaker statement.

The primary input from logarithmic deformation theory is the following result of Cheung, Fantini, Park, and Ulirsch~\cite[Theorem 1.1]{CFPU}. We record a proof in the context of stable maps theory, since it gives rise to a mild strengthening of the statement, see Remark~\ref{rem:CFPU}.

\begin{definition}\label{def: superabundant}
Let $\Psi:\Gamma \to \RR^r$ be balanced, piecewise linear map. Assume that $\Gamma$ is trivalent and the image of $[\Gamma]$ in $\mathcal M_g^{\trop}$ is the chain of $g$ cycles. The map $\Psi$ is said to be \textbf{superabundant} if there exists a cycle $\gamma_i$ such that $\Psi(\gamma_i)$ is contained in an affine hyperplane of $\RR^r$. The map $\Psi$ is said to be \textbf{non-superabundant} otherwise.
\end{definition}

Note that for general graphs $\Gamma$ that do not stabilize to the chain of cycles, there are more exotic ways for maps to be superabundant than having cycles contained in hyperplanes. See~\cite[Definition 4.1]{CFPU} for the general definition.

\begin{theorem}\label{unobstructed-lifting}
Let $\Psi:\Gamma\to \RR^r$ be a balanced piecewise linear map. Assume that $\Gamma$ is trivalent, the image $[\Gamma^{\mathrm{stab}}]$ of $[\Gamma]$ in $\mathcal M_g^{\trop}$ is a chain of $g$ cycles, and that $\Psi$ is non-superabundant. Then $\Psi$ is realizable. Precisely, there exists a valuation ring $R$ extending $\CC$ and a family of cures $\mathcal C$ over $\spec R$ with smooth generic fiber and equipped with a map
\[
\mathcal C\to Z
\]
whose tropicalization is equal to $\Psi$.
\end{theorem}

\begin{proof}
Let $\Delta$ be the complete fan in $\RR^r$ defining the toric variety $Z$.  After subdivision, this determines a prestable tropical map $\Gamma\to\Delta$ of combinatorial type $\Theta$. We momentarily assume that $[\Gamma]$ has integer edge lengths. Since $\Gamma$ is at most trivalent, by a standard construction, see~\cite[Proposition 2.8]{CFPU} or~\cite[Proposition 5.7]{NS06}, there is a degenerate logarithmic prestable map $[C'_0\to \mathscr Z]$ over $\spec(\NN\to \CC)$, where $\mathscr Z$ is a toric degeneration of $Z$ over $\CC[\![t]\!]$. By replacing this degeneration with a toroidal modification, we may assume that $\mathscr Z$ is a toroidal modification of the trivial degeneration $Z\times \spec \CC[\![t]\!]$. Composing $C'_0\to\mathscr Z$ with this map, we obtain a logarithmic prestable map $[C'_0\to Z]$ over $\spec(\NN\to \CC)$. By the universal property of minimality, there is a natural factorization
\[
\xymatrix{
C'_0 \ar[d] \ar[r] & C_0 \ar[r]\ar[d] & Z \\
\spec(\NN\to \CC) \ar[r] & \spec(Q\to \CC) &  ,}
\]
where $\spec(Q\to \CC)\to \mathscr L(Z)$ is minimal. The map $[C'_0\to Z]$ may now be discarded, and we work with $[C_0\to Z]$. When $\Gamma$ has general edge lengths, we simply choose a tropical map with integer edge lengths in the same combinatorial type as $\Psi$ to construct the minimal map above.

Let $[f_0:C_0\to Z]$ be the minimal logarithmic map determined above. We claim that $\mathscr L(Z)$ is logarithmically smooth at $[f_0]$. To see this, consider the logarithmic tangent-obstruction complex for $[f_0]$
\[
\cdots \to H^1(C_0,T_{C_0}^{\mathrm{log}}) \to H^1(C_0,f_0^\star T_{Z}^{\mathrm{log}})\to \mathrm{Ob}([f_0])\to 0.
\]
The first term encodes the infinitesimal logarithmic deformations of $C_0$, the second encodes obstructions to deforming the map fixing the curve, and the final term is the absolute obstructions of the map and curve. By~\cite[Section 4]{CFPU}, the non-superabundance of $\Psi$ forces surjectivity of the first arrow above. By exactness, it follows that $\mathrm{Ob}([f_0]) = 0$. A well-known argument implies that the space of maps is logarithmically smooth at $[f_0]$. We include it for completeness. By semicontinuity for the rank of the obstruction group, there is an open set $U$ in $\mathscr L(Z)$ containing $[f_0]$ where the obstruction space vanishes. Let $\mathbb E^\bullet$ denote the relative obstruction theory of $\mathscr L(Z)\to\mathscr L(\mathscr A_Z)$. It follows from unobstructedness of maps in $U$ that the complex $\mathbb E^\bullet$ has locally free cohomology in degree $0$ and vanishing cohomology in degree $-1$, on $U$. By the definition of a perfect obstruction theory, the logarithmic cotangent complex $\mathbb L^{\mathrm{log}}_{\mathscr L(Z)}$ also has locally free cohomology in degree $0$ and vanishing cohomology in degree $-1$, on $U$. It follows from this calculation that $\mathscr L(Z)$ is logarithmically smooth, and therefore formally locally isomorphic to the toric variety $\spec \CC[\![Q]\!]$. As a consequence, any map $\spec(\NN\to \CC)\to\mathscr L(Z)$ can be extended to a map from $\spec \CC[\![t]\!]\to\mathscr L(Z)$ where the latter is given the divisorial logarithmic structure at the closed point, and the generic point maps into the locus with trivial logarithmic structure. The result follows. 
\end{proof}

\begin{remark}\label{rem:CFPU}
The argument above gives a slightly stronger conclusion than the one recorded in~\cite{CFPU}. Specifically, the space of realizations of the tropical map has the structure of a fiber of a point under the projection to the skeleton of the analytification of a toroidal stack.This is precisely the situation for the moduli space of curves~\cite{ACP}. The result in~\cite{CFPU} assert that these fibers are nonempty. In the traditional context, the fibers of tropicalization are thought of as the \textit{initial degenerations}, and the argument here shows that the initial degenerations are smooth and of the expected dimension. Note that the trivalence in the above theorem is used only to construct the special fiber of the family; only a weaker hypothesis is needed to do so in \cite{CFPU}. 
\end{remark}

In order to prove Theorem~\ref{Thm:GenericLift}, it suffices to construct a non-superabundant map $\Psi : \Gamma \to \PP^r_{\trop}$ such that the hyperplane class is equivalent to the chosen divisor $D$.  This is done as follows. Let $D$ be a vertex avoiding divisor on $\Gamma$ with corresponding tableau $t$ and lattice path $\vec{p}$.  We perform a tropical modification to $\Gamma$, i.e. we append trees to $\Gamma$, in the following manner.  For each index $i$, and every $j$ that does not appear in the $i^{\textnormal{th}}$ column of the tableau, we attach an infinite ray to the $j^{\textnormal{th}}$ cycle based at $\langle \eta_j - p_{j-1} (i) \rangle_j$.  In other words, we attach an infinite ray at the point of the $j^{\textnormal{th}}$ cycle in the support of $D_i$.  By our assumptions that $D$ is vertex avoiding and that the edge lengths are generic, no two of these rays are based at the same point, and no ray is based at $v_j$ or $w_j$ for any $j$.  We also attach infinite rays based at $v_1$ and $w_g$.

Recall from the end of Section~\ref{Sec:TheGraph} that given a vertex avoiding divisor $D$ on $\Gamma$ of rank $r$, there exist \textit{canonical} piecewise linear functions $\psi_0,\ldots, \psi_r$, which carry $D$ to each of its distinguished representatives $D_0,\ldots, D_r$.

Consider the map $\Psi: \Gamma \to \mathbb{R}^r$ given by
\[
\Psi(p) = \left( \psi_0 (p) , \ldots , \psi_{r-1} (p) \right) .
\]
Note that the infinite rays of $\Gamma$ map to translates of the $r$ coordinate vectors $e_0, \ldots , e_{r-1}$ and $e_r = (-1,-1,\ldots,-1)$, so the map $\Psi$ extends to a map to tropical projective space.  More specifically, if an infinite ray is based at a point in the support of $D_i$, then it is a translate of the vector $e_i$.  We now consider the linear span of each cycle of $\Psi(\Gamma)$.

\begin{proposition}
\label{Prop:GenericSpan}
The image of each cycle in $\Psi(\Gamma)$ spans the entire vector space $\mathbb{R}^r$.
\end{proposition}

\begin{proof}
By the balancing condition, we see that the span of each cycle is equal to that of the infinite rays based at that cycle, plus that of the two bridges emanating from that cycle.  The infinite rays based at the $j^{\mathrm{th}}$ cycle are translates of the the vectors $e_i$ for all $i$ such that $j$ does not appear in the $i^{\mathrm{th}}$ column of the tableau.  By genericity of the edge lengths, the number $j$ can appear in at most one column of the tableau.  It follows that translates of at least $r$ of the vectors $e_0 , \ldots , e_r$ are based at the $j^{\textnormal{th}}$ cycle.  Since any $r$ of these vectors are linearly independent, we see that each cycle spans the entire space $\mathbb{R}^r$.
\end{proof}

\begin{proof}[Proof of Theorem~\ref{Thm:GenericLift}]
Let $D$ be a vertex avoiding divisor on $\Gamma$, and $\Psi : \Gamma \to \mathbb{P}_{\trop}^r$ the corresponding map.  By Proposition~\ref{Prop:GenericSpan}, the image of each cycle in $\Psi(\Gamma)$ spans $\mathbb{R}^r$.  Therefore, by Theorem~\ref{unobstructed-lifting}, there exists a smooth curve $C$ of genus $g$ over a non-archimedean extension of $\CC$ and a map $F: C \to \PP^r$ such that the tropicalization of $F$ is $\Psi$.  The map $F$ is given coordinatewise by
\[
F(p) = \left( f_0 (p), \ldots , f_r (p) \right),
\]
where $\trop (f_i) = \psi_i$.  Note that the functions $\psi_i$ are tropically independent by the remark following \cite[Notation~4.3]{JP16}, hence the functions $f_i$ are linearly independent, and the map $F$ is nondegenerate.  It follows that $F^\star \cO_{\PP^r} (1)$ is a divisor of rank at least $r$ on $C$ whose tropicalization is $D$.
\end{proof}

The main result of \cite{tropicalBN} is a new proof of the Brill--Noether theorem, using tropical techniques.  The argument proceeds by considering a chain of cycles $\Gamma$ with generic edge lengths, and letting $C$ be a curve over a valued field whose tropicalization is $\Gamma$.  Together with \cite[Theorem~6.9]{Gubler07}, Baker's specialization lemma \cite{Bak08} implies that $\dim W^r_d (C) \leq \dim W^r_d (\Gamma)$.  Since $\dim W^r_d (\Gamma) = \rho (g,r,d)$, we obtain the upper bound on $\dim W^r_d (C)$.  To obtain a lower bound, the authors use the results of Kempf and Kleiman--Laksov~\cite{Kempf, KleimanLaksov}, who establish this result by a study of the degeneracy locus of a map of vector bundles on the Picard variety. Using Theorem~\ref{Thm:GenericLift}, we can provide a proof that does not rely on the results of Kempf and Kleiman--Laksov.  Together with the results of~\cite{tropicalBN}, this completes a ``purely tropical'' proof of the Brill-Noether theorem.

\begin{theorem}[\textit{Brill--Noether theorem}]\label{thm: Kempf-Kleiman-Laksov}
Let $C$ be a general curve of genus $g$.  Then
\[
\dim W^r_d (C) = \rho (g,r,d).
\]
\end{theorem}

\begin{proof}
By the discussion above, we have $\dim W^r_d (C) \leq \rho (g,r,d)$.  It suffices to prove the reverse inequality.

Let $\mathscr C\to \cM_g$ be the universal curve and let $\cW^r_d$ be the universal $W^r_d$ over $\cM_g$, parameterizing pairs $(C,[D])$ where $C$ is a curve of genus $g$ and $[D]$ is a divisor class on $C$ of degree $d$ and rank at least $r$.  We will use a dimension estimate on $\widetilde \cW^r_d\subset \mathrm{Sym}^d(\mathscr C)$, the subset of the universal symmetric product, parametrizing pairs $(C,D)$ where $C$ is a curve of genus $g$ and $D$ is a divisor on $C$ of degree $d$ and rank at least $r$. This is the inverse image of $\cW^r_d$ in the map from the universal symmetric power to the universal Picard variety.

Let $[\mathrm{Sym}^d(\overline{\mathscr C})]$ be the $d^{\mathrm{th}}$ stack symmetric fibered power of the universal stable curve over $\Mbar_g$. As a stack, we may take an \'etale cover by the $d$-fold fiber product of the universal stable curve $\overline {\mathscr C}$ over $\Mbar_g$. This stack is logarithmically smooth (toroidal) in the \'etale topology and the forgetful morphism to $\Mbar_g$ is logarithmically smooth (toroidal).

Let $[\mathrm{Sym}^d(\overline{\mathscr C},\Gamma)]$ be the analytic domain of the analytification $[\mathrm{Sym}^d(\overline{\mathscr C})]^{\an}$ parametrizing families of pairs $(\mathcal C,\mathcal D)$ over a valuation ring such that the minimal skeleton $\Sigma(\mathcal C)$ is a chain of cycles. Let $\mathrm{Sym}^{d,\trop}(\Gamma)$ denote the symmetric product of the universal tropical curve, in the combinatorial type of the chain of cycles. Note that we can pass to a finite cover, so we may assume such a tropical universal curve exists in the category of cone complexes. The general structure results of~\cite[Section 6]{ACP} furnish a continuous tropicalization map
\[
\trop: [\mathrm{Sym}^{d,\an}(\overline{\mathscr C},\Gamma)]\to \mathrm{Sym}^{d,\trop}(\Gamma)
\]
sending a one-parameter family of pairs $(\mathcal C,\mathcal D)$ over a valuation ring to its image under the canonical retraction to the minimal skeleton $\Sigma(\mathcal C)$. Since the map to $\Mbar_g$ is logarithmically smooth, this tropicalization is compatible with the algebraic and tropical forgetful morphisms to the moduli space of curves.

It follows from Theorem~\ref{Thm:GenericLift} that $\trop(\widetilde \cW^{r,\an}_d)$ has dimension at least $3g-3+\rho(g,r,d)+r$. Given a substack $\mathcal Y$ of a logarithmically smooth stack $\mathcal X$, the dimension of $\mathcal Y$ is at least the dimension of $\trop(\mathcal Y)$, see~\cite[Theorem 1.1]{U15}. Thus, the dimension of $\widetilde \cW^r_d$ is at least $3g-3+\rho(g,r,d)+r$.  Combining Baker's specialization lemma with Lemma~\ref{Lem:ExactRank}, we see that there exists a $(C,D) \in \widetilde \cW^r_d$ such that $D$ has rank exactly $r$.  By semicontinuity of the rank, it follows that the relative dimension of $\widetilde{\mathcal W}^r_d\to \mathcal W^r_d$ is $r$ over an open subset of the target.  We conclude that $\cW^r_d$ has a component of dimension at least $3g-3+\rho (g,r,d)$.  By Theorem~\ref{Thm:GenericLift}, the image of this component in $\cM_g^{\text{trop}}$ has dimension $3g-3$.  Thus, the component dominates $\mathcal M_g$, and the fibers have dimension $\rho (g,r,d)$.
\end{proof}

The proof above would be more natural if there were a natural tropicalization map on $\mathcal W^r_d$ itself, or on the universal Picard variety. In lieu of this, we have had to work with the universal symmetric power.

\section{Extended example: Trigonal curves of genus $5$}
\label{Sec:Genus5}

As an example of what is to follow, we describe the Brill--Noether theory of a general trigonal curve of genus 5.  We show that such a curve has a divisor of degree 5 and rank 2, and a divisor of degree 8 and rank 4.  The results of this section can be obtained entirely via classical techniques. Indeed, these facts follow directly from the Riemann--Roch theorem. Nonetheless, we develop the theory using chains of cycles to foreshadow and motivate the approach of the later sections.

Let $\Gamma$ be a chain of 5 cycles with torsion profile
\[
\mu_i = \left\{ \begin{array}{ll}
3 & \textrm{if $i=3$,}\\
0 & \textrm{otherwise.}
\end{array} \right.
\]
The graph $\Gamma$ is trigonal with the divisor $E = 3v_3$ having rank 1.  Figure~\ref{Fig:Trigonal} depicts the divisor $E_1$ (in white) and the divisor $E_0$ (in black), together with the function $\varphi_0$ such that $\ddiv \varphi_0 = E_0 - E_1$.  We label the slopes of $\varphi_0$, going from left to right, on the bridges and bottom edges.  All other slopes are either 0 or 1.

\begin{figure}[h!]
\begin{tikzpicture}

\draw (0,0) circle (0.5);
\draw [ball color=white] (-0.5,0) circle (0.55mm);
\draw [ball color=white] (0.36,0.36) circle (0.55mm);
\draw (0,-0.75) node {\footnotesize 1};
\draw (0.5,0)--(1,0);
\draw (0.75,0.25) node {\footnotesize 2};
\draw (1.5,0) circle (0.5);
\draw [ball color=white] (1.5,0.5) circle (0.55mm);
\draw (1.5,-0.75) node {\footnotesize 2};
\draw (2,0)--(2.5,0);
\draw (2.25,0.25) node {\footnotesize 3};
\draw (3,0) circle (0.5);
\draw (3,-0.75) node {\footnotesize 2};
\draw (3.5,0)--(4,0);
\draw (3.75,0.25) node {\footnotesize 3};
\draw (4.5,0) circle (0.5);
\draw [ball color=black] (4.5,0.5) circle (0.55mm);
\draw (4.5,-0.75) node {\footnotesize 2};
\draw (5,0)--(5.5,0);
\draw (5.25,0.25) node {\footnotesize 2};
\draw (6,0) circle (0.5);
\draw [ball color=black] (5.64,0.36) circle (0.55mm);
\draw [ball color=black] (6.5,0) circle (0.55mm);
\draw (6,-0.75) node {\footnotesize 1};

\end{tikzpicture}
\caption{A trigonal chain of cycles of genus 5, together with its $g^1_3$.}
\label{Fig:Trigonal}
\end{figure}

The canonical divisor class is the unique divisor class $K \in \mathbf{T}(t)$, where $t$ is the tableau pictured in Figure~\ref{Fig:Canonical}.

\begin{figure}[!h]
\begin{tikzpicture}
\matrix[column sep=0.7cm, row sep = 0.7cm] {
\begin{scope}[node distance=0 cm,outer sep = 0pt]
	      \node[bsq] (11) at (2.5,1) {1};
	      \node[bsq] (12) [right = of 11] {2};
	      \node[bsq] (13) [right = of 12] {3};
	      \node[bsq] (14) [right = of 13] {4};
	      \node[bsq] (15) [right = of 14] {5};

\end{scope}
\\};
\end{tikzpicture}
\caption{The tableau $t$ corresponding to the canonical divisor.}
\label{Fig:Canonical}
\end{figure}
\noindent
In this case, the map $\Psi : \Gamma \to \PP^4_{\trop}$ defined in Section~\ref{Sec:GenericLift} is superabundant in the sense of Definition~\ref{def: superabundant}.  In particular, the image of the central cycle $\gamma_3$ only spans a plane in $\mathbb{R}^4$.  To see this, note that the bend locus of $\Psi$ on this cycle is supported at the three points $\langle -1 \rangle_3 , \langle 0 \rangle_3, \langle 1 \rangle_3$, and the linear span of 3 points is at most 2-dimensional.  For this reason, the argument of Section~\ref{Sec:GenericLift} does not apply, and we cannot show directly that the map $\Psi$ lifts.

To overcome the difficulty in showing that $\Psi$ lifts, we use the fact that the canonical embedding of a trigonal curve factors through a map to a rational normal surface scroll.  Concretely, any three points whose sum is in the class of the $g^1_3$ are collinear in the canonical embedding, and the union of these lines over all such divisors is the image of the scroll.  Instead of lifting the map $\Psi$ above, we will construct a map from $\Gamma$ to the tropicalization of this scroll, and show that this map lifts.  In the case of a trigonal curve of genus 5, the scroll in question is the first Hirzebruch surface $\mathbb{F}_1$.

To construct a map to $\mathbb{F}_1$, we consider the divisors $K(-1) = K-E$ and $K(-2) = K-2E$, whose corresponding tableaux $t(-1)$ and $t(-2)$ are depicted in Figure \ref{Fig:g25}.  Note that $t(-1)$ has 2 fewer columns and 1 more row than $t$.  The entries in the second row agree with those in the first row of $t$, shifted by a knight's move.  Similarly, $t(-2)$ has 2 fewer columns and 1 more row than $t(-1)$.  The entry in the last row agrees with that in the last row of $t(-1)$, shifted by a knight's move.

\begin{figure}[!h]
\begin{tikzpicture}
\matrix[column sep=0.7cm, row sep = 0.7cm] {
\begin{scope}[node distance=0 cm,outer sep = 0pt]
	      \node[bsq] (11) at (2.5,1) {1};
	      \node[bsq] (12) [right = of 11] {2};
	      \node[bsq] (13) [right = of 12] {3};
	      \node[bsq] (21) [below = of 11] {3};
	      \node[bsq] (22) [right = of 21] {4};
	      \node[bsq] (23) [right = of 22] {5};

	      \node[bsq] (11a) at (7.5,1) {1};
	      \node[bsq] (21a) [below = of 11a] {3};
	      \node[bsq] (31a) [below = of 21a] {5};

\end{scope}
\\};
\end{tikzpicture}
\caption{The tableaux $t(-1)$ and $t(-2)$ corresponding to $K(-1)$ and $K(-2)$.}
\label{Fig:g25}
\end{figure}

The divisor $K(-2)$ is rigid.  Its restriction to $\gamma_1 \cup \gamma_2$ is equal to that of $K_0$, and its restriction to $\gamma_3 \cup \gamma_4 \cup \gamma_5$ is equal to that of $K_4$.  Similarly, the divisor $K(-1)$ has rank 2, and the restriction of $K(-1)_i$ to $\gamma_1 \cup \gamma_2$ is equal to that of $K_i$, while its restriction to $\gamma_3 \cup \gamma_4 \cup \gamma_5$ is equal to that of $K_{i+2}$.  We let $\psi_i$ be a piecewise linear function such that $\ddiv \psi_i = K(-1)_i - (K(-2) + E_1)$.  These divisors and functions are pictured in Figure~\ref{Fig:SerreDual}.  As before, the divisor $K(-2) + E_1$ is pictured in white, and the divisor $K(-1)_i$ is pictured in black.

\begin{figure}[h!]
\begin{tikzpicture}

\draw (-2,0) node {$K(-2) + E_1$};
\draw (0,0) circle (0.5);
\draw [ball color=white] (-0.5,0) circle (0.55mm);
\draw [ball color=white] (0.36,0.36) circle (0.55mm);
\draw (0.5,0)--(1,0);
\draw (1.5,0) circle (0.5);
\draw [ball color=white] (1.5,0.5) circle (0.55mm);
\draw [ball color=white] (1.14,0.36) circle (0.55mm);
\draw (2,0)--(2.5,0);
\draw (3,0) circle (0.5);
\draw (3.5,0)--(4,0);
\draw (4.5,0) circle (0.5);
\draw [ball color=white] (4.86,0.36) circle (0.55mm);
\draw (5,0)--(5.5,0);
\draw (6,0) circle (0.5);

\draw (-2,-2) node {$K(-1)_0$};
\draw (0,-2) circle (0.5);
\draw [ball color=white] (-0.5,-2) circle (0.55mm);
\draw [ball color=white] (0.36,-1.64) circle (0.55mm);
\draw (0,-2.75) node {\footnotesize 1};
\draw (0.5,-2)--(1,-2);
\draw (0.75,-1.75) node {\footnotesize 2};
\draw (1.5,-2) circle (0.5);
\draw [ball color=white] (1.5,-1.5) circle (0.55mm);
\draw [ball color=black] (1.14,-1.64) circle (0.55mm);
\draw (1.5,-2.75) node {\footnotesize 2};
\draw (2,-2)--(2.5,-2);
\draw (2.25,-1.75) node {\footnotesize 3};
\draw (3,-2) circle (0.5);
\draw (3,-2.75) node {\footnotesize 2};
\draw (3.5,-2)--(4,-2);
\draw (3.75,-1.75) node {\footnotesize 3};
\draw (4.5,-2) circle (0.5);
\draw [ball color=white] (4.86,-1.64) circle (0.55mm);
\draw [ball color=black] (4.14,-1.64) circle (0.55mm);
\draw (4.5,-2.75) node {\footnotesize 2};
\draw (5,-2)--(5.5,-2);
\draw (5.25,-1.75) node {\footnotesize 3};
\draw (6,-2) circle (0.5);
\draw [ball color=black] (6,-1.5) circle (0.55mm);
\draw [ball color=black] (6.5,-2) circle (0.55mm);
\draw (6,-2.75) node {\footnotesize 2};
\draw (6.75,-2) node {\footnotesize 2};

\draw (-2,-4) node {$K(-1)_1$};
\draw (0,-4) circle (0.5);
\draw [ball color=black] (-0.5,-4) circle (0.55mm);
\draw [ball color=black] (0.36,-3.64) circle (0.55mm);
\draw (0,-4.75) node {\footnotesize 0};
\draw (0.5,-4)--(1,-4);
\draw (0.75,-3.75) node {\footnotesize 0};
\draw (1.5,-4) circle (0.5);
\draw [ball color=white] (1.5,-3.5) circle (0.55mm);
\draw [ball color=white] (1.14,-3.64) circle (0.55mm);
\draw (1.5,-4.75) node {\footnotesize 1};
\draw (2,-4)--(2.5,-4);
\draw (2.25,-3.75) node {\footnotesize 2};
\draw (3,-4) circle (0.5);
\draw [ball color=black] (3,-3.5) circle (0.55mm);
\draw (3,-4.75) node {\footnotesize 1};
\draw (3.5,-4)--(4,-4);
\draw (3.75,-3.75) node {\footnotesize 1};
\draw (4.5,-4) circle (0.5);
\draw [ball color=white] (4.86,-3.64) circle (0.55mm);
\draw (4.5,-4.75) node {\footnotesize 1};
\draw (5,-4)--(5.5,-4);
\draw (5.25,-3.75) node {\footnotesize 2};
\draw (6,-4) circle (0.5);
\draw [ball color=black] (5.64,-3.64) circle (0.55mm);
\draw [ball color=black] (6.5,-4) circle (0.55mm);
\draw (6,-4.75) node {\footnotesize 1};

\draw (-2,-6) node {$K(-1)_2$};
\draw (-0.75,-6) node {\footnotesize 2};
\draw (0,-6) circle (0.5);
\draw [ball color=black] (-0.5,-6) circle (0.55mm);
\draw [ball color=white] (0.36,-5.64) circle (0.55mm);
\draw [ball color=black] (0,-5.5) circle (0.55mm);
\draw (0,-6.75) node {\footnotesize -1};
\draw (0.5,-6)--(1,-6);
\draw (0.75,-5.75) node {\footnotesize -1};
\draw (1.5,-6) circle (0.5);
\draw [ball color=white] (1.5,-5.5) circle (0.55mm);
\draw [ball color=white] (1.14,-5.64) circle (0.55mm);
\draw [ball color=black] (1.86,-5.64) circle (0.55mm);
\draw (1.5,-6.75) node {\footnotesize 0};
\draw (2,-6)--(2.5,-6);
\draw (2.25,-5.75) node {\footnotesize 0};
\draw (3,-6) circle (0.5);
\draw (3,-6.75) node {\footnotesize 0};
\draw (3.5,-6)--(4,-6);
\draw (3.75,-5.75) node {\footnotesize 0};
\draw (4.5,-6) circle (0.5);
\draw [ball color=black] (4.86,-5.64) circle (0.55mm);
\draw (4.5,-6.75) node {\footnotesize 0};
\draw (5,-6)--(5.5,-6);
\draw (5.25,-5.75) node {\footnotesize 0};
\draw (6,-6) circle (0.5);
\draw (6,-6.75) node {\footnotesize 0};

\end{tikzpicture}
\caption{Divisors in the linear system $K_{\Gamma} - g^1_3$.}
\label{Fig:SerreDual}
\end{figure}

We now define a function $\Psi : \Gamma \to \mathbb{F}_1^{\trop}$ by $\Psi (p) = (\varphi_0 (p), \psi_1 (p))$.  The image of this map is pictured in Figure~\ref{Fig:MapToScroll}.

\begin{figure}[!h]
\begin{tikzpicture}
\matrix[column sep=0.7cm, row sep = 0.7cm] {
\begin{scope}[node distance=0 cm,outer sep = 0pt]

          \draw (0.5,0)--(1,0);
          \draw (0.5,0)--(-3.5,-4);
          \draw (0.5,0)--(0.5,8);
          \draw (0,8) node {\footnotesize 2};
          \draw (-3.5,-3.5) node {\footnotesize 2};

          \draw (1,0)--(2,0);

          \draw (2,0)--(4,1);
          \draw (2,-1)--(4,1);
          \draw (2,0)--(2,-1);
          \draw (2,-1)--(1.5,-2);
          \draw (1.5,-2)--(1.5,-4);
          \draw (1.5,-2)--(-0.5,-4);

          \draw (4,1)--(5.5,2);

          \draw (5.5,2)--(7.5,3);
          \draw (5.5,2)--(6.5,3);
          \draw (6.5,3)--(7.5,3);
          \draw (6.5,3)--(6.5,8);

          \draw (7.5,3)--(9,3.5);

          \draw (9,3.5)--(10,4.5);
          \draw (9,3.5)--(10,3.5);
          \draw (10,3.5)--(10,4.5);
          \draw (10,3.5)--(10.5,3);
          \draw (10.5,3)--(10.5,-4);
          \draw (10.5,3)--(13,3);

          \draw (10,4.5)--(11,5.5);

          \draw (11,5.5)--(11.5,6);
          \draw (11.5,6)--(13,6);
          \draw (11.5,6)--(11.5,8);
          \draw (12,8) node {\footnotesize 2};
          \draw (13,6.5) node {\footnotesize 2};

\end{scope}
\\};
\end{tikzpicture}
\caption{The image of $\Gamma$ in $\mathbb{F}_1^{\trop}$. The map is manifestly superabundant, since the image has genus $3$ and $\Gamma$ is a chain of $5$ cycles.}
\label{Fig:MapToScroll}
\end{figure}

The map $\Psi$ is still superabundant -- it maps the first and last cycles $\gamma_1$ and $\gamma_5$ to line segments.  In this situation, we may nevertheless apply the results of Section~\ref{Sec:Def} to show that the map lifts provided that a certain combinatorial condition \textit{(naive well-spacedness)} holds on the edge lengths of the trees attached to $\Gamma$.  We conclude that there exists a curve $C$ of genus 5 over a non-archimedean field and a map $F : C \to \mathbb{F}_1$ specializing to $\Psi$.  Taking $D$ to be the preimage on $C$ of the $(-1)$-curve on $\mathbb{F}_1$, we see that $D$ is effective, and since $D$ specializes to $K(-2)$, it has rank exactly 0.  The divisor $D+g^1_3$ is then the pullback of the hyperplane class from the map $\mathbb{F}_1 \to \PP^2$ that contracts the $(-1)$-curve.  From Figure~\ref{Fig:MapToScroll}, we see that the image of this map is not contained in a line, and $D+g^1_3$ therefore has rank at least 2. By the basepoint free pencil trick, it follows that $D+2g^1_3$ has rank at least 4.

\section{Maps to scrolls}
\label{Sec:Scrolls}

We now attempt to generalize the argument of Section~\ref{Sec:GenericLift} to special chains of cycles.  As seen in Section~\ref{Sec:Genus5}, the argument does not generalize directly, for the following reason:  if the symbol $j$ appears in the tableau and the $j^{\textnormal{th}}$ torsion order $\mu_j$ is nonzero, then the image of the $j^{\textnormal{th}}$ cycle under the map $\Psi$ spans a linear space of dimension at most $\mu_j-1$.  As such, if $r \geq \mu_j$, then Theorem~\ref{unobstructed-lifting} does not apply.  Our main observation is that, if a general curve of fixed gonality admits a divisor of high rank, then the corresponding map to $\PP^r$ factors through a map to a lower dimensional scroll.  Rather than constructing a map from $\Gamma$ to tropical projective space, we instead construct a map to a tropical scroll.  In this section, we describe the basic theory of maps to scrolls.

Fix integers $a \geq 1$, $b \geq 0$.  To simplify notation, we will write $\mathbb{S}(a,b)$ for the balanced scroll
\[
\mathbb{S}(a,b) := \PP ( \cO_{\PP^1}^{\oplus a} \oplus \cO_{\PP^1}(1)^{\oplus b} ) .
\]
We write $\pi : \mathbb{S}(a,b) \to \PP^1$ for the natural map to $\PP^1$, and $n = a+b$ for the dimension of $\mathbb{S}(a,b)$.  The scroll $\mathbb{S}(a,b)$ is a toric variety, and its fan $\Sigma (a,b)$ is described in \cite[Example~7.3.5]{CoxLittleSchenk}.  We recall this description.  In $\mathbb{R}^n$, we let $u_1$ denote the first coordinate vector and $e_1 , \ldots , e_{n-1}$ the remaining coordinate vectors.  We let $e_0 = - \sum_{i=1}^{n-1} e_i$ and $u_0 = -u_1 -\sum_{i=b}^{n-1} e_i$.  The maximal cones of $\Sigma (a,b)$ are the $2n$ cones spanned by one of the two vectors $u_0, u_1$ and all but one of the vectors $e_0, \ldots , e_{n-1}$.\footnote{In the notation of \cite{CoxLittleSchenk}, the fan described here is actually that of the projectivization of the dual vector bundle.  We prefer our projectivizations to parameterize one-dimensional quotients, rather than one-dimensional subspaces.}

For the sake of clarity, we describe the degenerate cases.  If $b=0$, then the scroll $S(a,0)$ is a product of projective spaces, and the vector $u_0$ is simply $-u_1$.  If both $b=0$ and $a=1$, then the scroll $S(1,0)$ is just $\PP^1$, and the only rays of the fan are $u_1$ and $u_0 = -u_1$.

As shown by Cox, every smooth toric variety $X_{\Sigma}$ can be viewed as a fine moduli space for collections of line bundles and sections with trivialization data.

\begin{definition}
For any scheme $Y$, a \textit{$\Sigma$-collection} on $Y$ is the data:
\begin{enumerate}
\item  for each $\rho \in \Sigma (1)$, a line bundle $L_{\rho}$ on $Y$ and a distinguished section $u_{\rho} \in H^0 (Y, L_{\rho})$.
\item  for each $m \in M$, a trivialization $\varphi_m : \bigotimes_{\rho \in \Sigma (1)} L_{\rho}^{\langle \rho , m \rangle} \to \cO_Y$ such that $\varphi_{m+m'} = \varphi_m \otimes \varphi_{m'}$ and, for each $y \in Y$, there exists a maximal cone $\sigma$ such that $u_{\rho} (y) \neq 0$ for all $\rho \nsubseteq \sigma$.
\end{enumerate}
\end{definition}

Families of such collections are defined in the natural manner, and the scheme $X_{\Sigma}$ is a fine moduli space for isomorphism classes of $\Sigma$-collections.  In other words, a map $Y \to X_{\Sigma}$ is given by the same data as a $\Sigma$-collection on $Y$.

In the case where $C$ is a curve and $\Sigma = \Sigma (a,b)$ is the fan of a balanced scroll, a $\Sigma$-collection on $C$ simplifies to the following data:
\begin{enumerate}
\item  Two line bundles $D$ and $E$;
\item  Two sections $s_0 , s_1 \in H^0 (C,E)$ that do not simultaneously vanish at any point.
\item  A collection of $b$ sections $f_0 , \ldots , f_{b-1} \in H^0(C,D-E)$ and a collection of $a$ sections $f_b , \ldots , f_{n-1} \in H^0(C,D)$ that do not simultaneously vanish at any point.
\end{enumerate}
We say that a map $F : C \to \mathbb{S}(a,b)$ is \emph{nondegenerate} if it is nondegenerate under its composition with the natural embedding of $\mathbb{S}(a,b)$ in $\PP^{n+b-1}$ given by $\cO_{\mathbb{S}(a,b)} (1) \otimes \pi^\star \cO_{\PP^1} (1)$.  Equivalently, the map $F$ is nondegenerate if the sections
\[ s_0 f_0 , \ldots , s_0 f_{b-1} , s_1 f_0 , \ldots s_1 f_{b-1} , f_b , \ldots , f_{n-1} \in H^0(C,D) \]
are linearly independent.

\begin{proposition}
\label{Prop:MapToScroll}
Let $C$ be a curve of gonality $k$, and let $E$ denote a divisor of degree $k$ and rank 1.  If $C$ admits a nondegenerate map to $\mathbb{S}(a,b)$ with $\pi^\star \cO_{\PP^1} (1) = E$, then there exists a divisor $D$ on $C$ of rank at least $n+b-1$ such that $D-E$ has rank at least $b-1$.
\end{proposition}

\begin{proof}
Suppose that $C$ admits a nondegenerate map to a scroll $\mathbb{S}(a,b)$.  Then, by the description above, there exist $n+b$ linearly independent sections in $H^0 (C,D)$ and $b$ linearly independent sections in $H^0 (C,D-E)$, so $D$ has rank at least $n+b-1$ and $D-E$ has rank at least $b-1$.
\end{proof}

The following proposition is a partial converse to Proposition~\ref{Prop:MapToScroll}.

\begin{proposition}
\label{Prop:PartConverse}
Let $C$ be a curve of gonality $k$, and let $E$ denote a divisor of degree $k$ and rank 1.  The curve $C$ admits a nondegenerate map to $\mathbb{S}(a,b)$ with $\pi^\star \cO_{\PP^1} (1) = E$ if there exists a divisor $D$ on $C$ with the following properties:
\begin{enumerate}
\item  $r(D) \geq n+b-1$;
\item  $r(D-E) = b-1$;
\item  $r(D-2E) = -1$.
\end{enumerate}
\end{proposition}

\begin{proof}
Let $s_0 , s_1$ be a basis for $H^0 (C,E)$, and let $f_0 , \ldots , f_{b-1}$ be a basis for $H^0 (C,D-E)$.  By the basepoint free pencil trick, we have an exact sequence
\[
0 = H^0 (C,D-2E) \to H^0 (C,E) \otimes H^0 (C,D-E) \to H^0 (C,D) .
\]
It follows that the righthand map is injective, so the functions $s_i f_j \in H^0 (C,D)$ are linearly independent.  Letting $f_b , \ldots , f_{n-1} \in H^0 (C,D)$ be an extension of this linearly independent set, we obtain a nondegenerate $\Sigma (a,b)$-collection on $C$.  The result follows.
\end{proof}

Geometrically, this is very similar to the case of trigonal curves described in Section~\ref{Sec:Genus5}.  If we consider the image of $C$ in $\PP^{n+b-1}$ under the (possibly incomplete) linear series in $\vert D \vert$ described above, then any $k$ points whose sum is in the class of $E$ span a projective space of dimension $n-1$.  The union of these projective spaces is the image of the scroll $\mathbb{S}(a,b)$.

In the situation of Proposition~\ref{Prop:PartConverse}, we can obtain divisors of large rank simply by adding multiples of the divisor $E$.

\begin{corollary}
\label{Cor:AllM}
Let $C, D, E$ be as in Proposition~\ref{Prop:PartConverse}.  Then $h^0 (D+mE) \geq (m+1)n+b-1$ for all $m\geq 0$.
\end{corollary}

\begin{proof}
We prove, by induction on $m$, that
\[
h^0 (C,D+mE) \geq h^0 (C,D+(m-1)E) + n.
\]
The case $m=0$ is given.  By the basepoint free pencil trick, we have the exact sequence
\[
0 \to H^0 (C,D+(m-2)E) \to H^0 (C,E) \otimes H^0 (C,D+(m-1)E) \to H^0 (C,D+mE) .
\]
Thus,
\begin{align*}
h^0 (C,D+mE) & \geq 2h^0 (C,D+(m-1)E) - h^0 (C,D+(m-2)E) \\
& \geq h^0 (X,D+(m-1)E) + n,
\end{align*}
where the second inequality holds by the inductive hypothesis.
\end{proof}

\section{Scrollar tableaux}
\label{Sec:ScrollarTableaux}

In this section and in Section~\ref{Sec:SpecialLift}, we fix an integer $k \geq 2$ and consider a chain of cycles $\Gamma$ with torsion profile
\[
\mu_i = \left\{ \begin{array}{ll}
0 & \textrm{if $i \leq k-1$ or $i \geq g-k+2$,}\\
k & \textrm{otherwise.}
\end{array} \right.
\]
In this section, we analyze in detail the combinatorics of certain divisors on $\Gamma$.  There are many reasons to think that $\Gamma$ should behave like a general curve of gonality $k$.  Note in particular that the moduli space of such graphs has dimension $2g-5+2k$, equal to that of the $k$-gonal locus in $\cM_g$, and each admits a harmonic degree $k$ morphism to $\RR$ with vanishing Riemann-Hurwitz obstruction~\cite{ABBR,Pfl16b}.  Throughout, we write $E = kv_k$, which is a divisor on $\Gamma$ of degree $k$ and rank 1.

As in Section~\ref{Sec:TheGraph}, we define the divisor $E_i$ to be the unique divisor equivalent to $E$ such that $E_i - iv_1 - (1-i)w_g$ is effective.  So
\[
E_1 = v_1 + \sum_{j=1}^{k-1} \langle j \rangle_j
\]
and
\[
E_0 = w_g + \sum_{j=g-k+2}^g \langle j-(g+2) \rangle_j .
\]
We let $\varphi_i$ be a piecewise linear function such that $\ddiv \varphi_i = E_i - E_1$.

Given this torsion profile $\vec\mu$, we consider $\vec\mu$-displacement tableaux of the following form.

\begin{definition}
Fix two integers $a>0, b \geq 0$ such that $n=a+b<k$.  We define a tableau $t$ to be \textbf{scrollar} of type $(a,b)$ if
\begin{enumerate}
\item  $t(x,y) = t(x',y')$ if and only if there exists an integer $\ell$ such that both $x'-x = \ell n$ and $y'-y = \ell (n-k)$, and
\item  the number of columns is at least $n$ and is congruent to $b \pmod{n}$.
\end{enumerate}
\end{definition}

An example of such a tableau appears in Figure~\ref{Fig:Scrollar}.  Note that the boxes in the first $n$ columns necessarily contain distinct numbers, as do the boxes in the last $k-n$ rows.  The remaining boxes contain symbols that appear in this L-shaped region.

\begin{figure}[!h]
\begin{tikzpicture}
\matrix[column sep=0.7cm, row sep = 0.7cm] {
\begin{scope}[node distance=0 cm,outer sep = 0pt]
          \fill[fill=gray] (2.15,1.35) rectangle (4.25,-2.15);
          \fill[fill=gray] (4.25,-0.75) rectangle (7.75,-2.15);
	      \node[bsq] (11) at (2.5,1) {1};
	      \node[bsq] (12) [right = of 11] {2};
	      \node[bsq] (13) [right = of 12] {3};
	      \node[bsq] (14) [right = of 13] {7};
	      \node[bsq] (15) [right = of 14] {8};
	      \node[bsq] (16) [right = of 15] {9};
	      \node[bsq] (17) [right = of 16] {13};
	      \node[bsq] (18) [right = of 17] {14};
	      \node[bsq] (21) [below = of 11] {4};
	      \node[bsq] (22) [right = of 21] {5};
	      \node[bsq] (23) [right = of 22] {6};
	      \node[bsq] (24) [right = of 23] {10};
	      \node[bsq] (25) [right = of 24] {11};
	      \node[bsq] (26) [right = of 25] {12};
	      \node[bsq] (27) [right = of 26] {16};
	      \node[bsq] (28) [right = of 27] {17};
	      \node[bsq] (31) [below = of 21] {7};
	      \node[bsq] (32) [right = of 31] {8};
	      \node[bsq] (33) [right = of 32] {9};
	      \node[bsq] (34) [right = of 33] {13};
	      \node[bsq] (35) [right = of 34] {14};
	      \node[bsq] (36) [right = of 35] {15};
	      \node[bsq] (37) [right = of 36] {19};
	      \node[bsq] (38) [right = of 37] {20};
	      \node[bsq] (41) [below = of 31] {10};
	      \node[bsq] (42) [right = of 41] {11};
	      \node[bsq] (43) [right = of 42] {12};
	      \node[bsq] (44) [right = of 43] {16};
	      \node[bsq] (45) [right = of 44] {17};
	      \node[bsq] (46) [right = of 45] {18};
	      \node[bsq] (47) [right = of 46] {22};
	      \node[bsq] (48) [right = of 47] {23};
	      \node[bsq] (51) [below = of 41] {13};
	      \node[bsq] (52) [right = of 51] {14};
	      \node[bsq] (53) [right = of 52] {15};
	      \node[bsq] (54) [right = of 53] {19};
	      \node[bsq] (55) [right = of 54] {20};
	      \node[bsq] (56) [right = of 55] {21};
	      \node[bsq] (57) [right = of 56] {24};
	      \node[bsq] (58) [right = of 57] {25};

\end{scope}
\\};
\end{tikzpicture}
\caption{An example of a scrollar tableau $t$ of type $(1,2)$, with $k=5$.}
\label{Fig:Scrollar}
\end{figure}

If $t$ is a scrollar tableau of type $(a,b)$ with $r+1$ columns and $s$ rows, define $t(-1)$ to be the tableau with $r+1-n$ columns and $s+(k-n)$ rows defined as follows:
\[
t(-1)(x,y) = \left\{ \begin{array}{ll}
t(x,y) & \textrm{if $y<s$,}\\
t(x+n,y-(k-n)) & \textrm{if $y \geq s$.}
\end{array} \right.
\]
An example of such a tableau appears in Figure~\ref{Fig:Hat}.  Note that $t(-1)$ is the union of the first $r+1-n$ columns of $t$ and the last $r+1-n$ columns of $t$, translated to the left $n$ units and down $k-n$ units.  

\begin{figure}[!h]
\begin{tikzpicture}
\matrix[column sep=0.7cm, row sep = 0.7cm] {
\begin{scope}[node distance=0 cm,outer sep = 0pt]
          \fill[fill=gray] (2.15,1.35) rectangle (4.25,-3.55);
          \fill[fill=gray] (4.25,-2.15) rectangle (5.65,-3.55);
	      \node[bsq] (11) at (2.5,1) {1};
	      \node[bsq] (12) [right = of 11] {2};
	      \node[bsq] (13) [right = of 12] {3};
	      \node[bsq] (14) [right = of 13] {7};
	      \node[bsq] (15) [right = of 14] {8};
	      \node[bsq] (21) [below = of 11] {4};
	      \node[bsq] (22) [right = of 21] {5};
	      \node[bsq] (23) [right = of 22] {6};
	      \node[bsq] (24) [right = of 23] {10};
	      \node[bsq] (25) [right = of 24] {11};
	      \node[bsq] (31) [below = of 21] {7};
	      \node[bsq] (32) [right = of 31] {8};
	      \node[bsq] (33) [right = of 32] {9};
	      \node[bsq] (34) [right = of 33] {13};
	      \node[bsq] (35) [right = of 34] {14};
	      \node[bsq] (41) [below = of 31] {10};
	      \node[bsq] (42) [right = of 41] {11};
	      \node[bsq] (43) [right = of 42] {12};
	      \node[bsq] (44) [right = of 43] {16};
	      \node[bsq] (45) [right = of 44] {17};
	      \node[bsq] (51) [below = of 41] {13};
	      \node[bsq] (52) [right = of 51] {14};
	      \node[bsq] (53) [right = of 52] {15};
	      \node[bsq] (54) [right = of 53] {19};
	      \node[bsq] (55) [right = of 54] {20};
	      \node[bsq] (61) [below = of 51] {16};
	      \node[bsq] (62) [right = of 61] {17};
	      \node[bsq] (63) [right = of 62] {18};
	      \node[bsq] (64) [right = of 63] {22};
	      \node[bsq] (65) [right = of 64] {23};
	      \node[bsq] (71) [below = of 61] {19};
	      \node[bsq] (72) [right = of 71] {20};
	      \node[bsq] (73) [right = of 72] {21};
	      \node[bsq] (74) [right = of 73] {24};
	      \node[bsq] (75) [right = of 74] {25};

\end{scope}
\\};
\end{tikzpicture}
\caption{The tableau $t(-1)$, where $t$ is the tableau depicted in Figure~\ref{Fig:Scrollar}.}
\label{Fig:Hat}
\end{figure}

Our interest in these tableaux is illustrated by the following proposition.

\begin{proposition}
\label{Prop:HatT}
Let $t$ be a scrollar tableau of type $(a,b)$, and let $D \in \mathbf{T}(t)$.  Then $D-E \in \mathbf{T}(t(-1))$.  Moreover, for all $0\leq i\leq r-n$, $D-E$ is equivalent to the divisor $D(-1)_i$ that agrees with $D_i$ on $\bigcup_{j<k} \gamma_j$ and agrees with $D_{i+n}$ on $\bigcup_{j \geq k} \gamma_j \smallsetminus \{ v_k \}$.
\end{proposition}

\begin{proof}
Suppose that $t$ has $r+1$ columns and $s$ rows. We alert the reader that $r+1$ is at least $n$ by the definition of a scrollar tableau.  Then the degree of any divisor in $\mathbf{T}(t(-1))$ is equal to
\[
g+(r-n)-(s+k-n) = (g+r-s)-k = \deg (D) - k.
\]
Next, recall that the function $\tilde{\xi}_j$ is a linear function on divisors.  It is easy to see that
\[
\tilde{\xi}_j (D-kv_k) = \left\{ \begin{array}{ll}
\tilde{\xi}_j (D) & \textrm{if $j<k$,}\\
\tilde{\xi}_j (D)-k & \textrm{if $j \geq k.$}
\end{array} \right.
\]
Now, if $k \leq j \leq g-k+1$, then $\langle \xi_j - k \rangle_j = \langle \xi_j \rangle_j$.  We therefore see that $D-E \in \mathbf{T}(t(-1))$. It follows that $D-E$ is equivalent to the divisor $D(-1)_i$ that agrees with $D_i$ on $\bigcup_{j<k} \gamma_j$ and agrees with $D_{i+n}$ on $\bigcup_{j \geq k} \gamma_j \smallsetminus \{ v_k \}$.
\end{proof}

Given a scrollar tableau $t$, we write $t(-m)$ for the tableau obtained by successively applying $m$ times the operation $t \mapsto t(-1)$.  For $D \in \mathbf{T}(t)$, we write $D(-m) = D-mE \in \mathbf{T}(t(-m))$.  By definition, there exists an integer $m = \lfloor \frac{r+1}{n} \rfloor$ such that $t(-m)$ is a standard Young tableau with $b$ columns.  Note that $m > 0$, since in the case where $t$ is a standard Young tableau, we have $n = a = r+1$.

\begin{corollary}
\label{Cor:Ranks}
Let $t$ be a scrollar tableau of type $(a,b)$ with $r+1$ columns, let $D \in \mathbf{T}(t)$ be sufficiently general, and let $m = \lfloor \frac{r+1}{n} \rfloor$.  Then $D(-m)$ is vertex avoiding.  In particular,
\begin{enumerate}
\item  $r(D(-m)) = b$;
\item  $r(D(-m-1)) = -1$.
\end{enumerate}
\end{corollary}

\begin{proof}
Since $D$ is sufficiently general, we may restrict our attention to those cycles $\gamma_j$ such that $j$ appears in the tableau $t$.  Note that the tableau $t(-m+1)$ contains the same symbols as $t$.  Suppose that the number $j$ appears in the nonempty set of columns $S$ of the tableau $t(-m+1)$, and let $\vec{p}$ be the lingering lattice path corresponding to the tableau $t(-m+1)$.  Since $p_j (i) > p_j (i+1)$ for all $i$ and $p_j (b-1) = k+ p_j(n+b-1)$, we see that the values $p_j (i)$ are distinct mod $k$ for $b \leq i \leq n+b-1$.  It follows that
\[
\eta_j -p_{j-1} (i) \neq 0 \pmod{k}
\]
for all $i \notin S$.  Similarly, since the values of $p_{j+1} (i)$ are distinct mod $k$, we have
\[
\eta_j - p_{j-1} (i) \neq -1 \pmod{k}
\]
for all $i$.  By Proposition~\ref{Prop:HatT}, we see that $D(-m)$ is vertex avoiding.

To see that $r(D(-m-1)) = -1$, note that no box in row $(k-n)-1$ of $t$ may contain an integer smaller than $k$.  It follows that
\[
p_{k-1} (0) \leq b + k-1 <  2k .
\]
From this, we see that the $v_k$-reduced divisor equivalent to $D(-m+1)$ has $\deg_{v_k} D < 2k$.  It follows that the $v_k$-reduced divisor equivalent to $D-(m+1)kv_k$ has negative degree at $v_k$.  Thus, $D-(m+1)E$ has negative rank.
\end{proof}

By \cite[Lemma~3.4]{Pfl16b}, the tableaux $t$ that maximize the dimension of $\mathbf{T}(t)$ are scrollar tableaux.  We reproduce the dimension compuatation below for completeness.

\begin{proposition}
\label{Prop:Remainder}
Fix positive integers $r$ and $d$.  Let $t$ be a scrollar tableau of type $(a,b)$ with $r+1$ columns and $s = g-d+r$ rows.  Let $\ell = r+1-n$.  Then
\[
\dim \mathbf{T}(t) = \rho (g,r-\ell,d) - \ell k .
\]
In particular, there exists a scrollar tableau $t$ with $r+1$ rows and $s$ columns such that
\[
\dim \mathbf{T}(t) = {\rho}_k (g,r,d) .
\]
\end{proposition}

\begin{proof}
It suffices to show that the number of distinct symbols in the tableau $t$ is
\[
(r+1-\ell)(s-\ell) + \ell k = n(s-(r+1-n)) + (r+1-n)k.
\]
Note that each box in the first $n$ columns of $t$ contains a distinct symbol, as does each box in the last $k-n$ rows.  All other symbols appearing in the tableau $t$ occur in one of the boxes in this L-shaped region.  It follows that the number of distinct symbols in $t$ is
\begin{eqnarray*}
ns + (r+1)(k-n) - n(k-n) &=& n(s-(r+1-n)) + (r+1-n)k.
\end{eqnarray*}
\end{proof}

Let $t$ be a scrollar tableau of type $(a,b)$ with $r+1$ columns, and let $D \in \mathbf{T}(t)$ be sufficiently general.  Let $m = \lfloor \frac{r+1}{n} \rfloor$.  As in Section~\ref{Sec:TheGraph}, if $b \neq 0$, for $0 \leq i < b$ we let $D(-m)_i$ be the unique divisor equivalent to $D-mE$ such that $D(-m)_i-iv_1 - (b-1-i)w_g$ is effective, and let $\psi_i$ be a piecewise linear function such that $\ddiv \psi_i = D(-m)_i - D(-m)_{b-1}$.  If $b=0$, then $D(-m) = 0$ and $\psi$ is the constant function.

Similarly, for $b \leq i < n$, we let $D(-m+1)_i$ be the unique divisor equivalent to $D-(m-1)E$ such that $D(-m+1)_i - iv_1 - (n+b-1-i)w_g$ is effective, and let $\psi_i$ be a piecewise linear function such that $\ddiv \psi_i = D(-m+1)_i - (D(-m)_{b-1} + E_1)$.  Note that the poles of $D(-m+1)_i$ are not along $D(-m+1)_{n+b-1}$ as in Section~\ref{Sec:TheGraph}, but instead along $D(-m)_{b-1} + E_1$.  The reason for this choice is that the functions $\psi_i$ will later be used to define a tropical $\Sigma (a,b)$-collection on $\Gamma$.

\begin{lemma}
\label{Lem:Independent}
The set of functions
\[
\{\varphi_0 + \psi_0 , \ldots , \varphi_0 + \psi_{b-1}\} \cup \{\varphi_1 + \psi_0 , \ldots , \varphi_1 + \psi_{b-1}\} \cup\{ \psi_b , \ldots , \psi_{n-1}\}
\]
is tropically independent in the sense of \cite[Definition 3.1]{JP14}.
\end{lemma}

\begin{proof}
It suffices to show that these functions have distinct slopes $\sigma_{k-1}$ along the bridge $\beta_{k-1}$.  If $\vec{p}$ is the lingering lattice path corresponding to the tableau $t(-m+1)$, then we have
\begin{align*}
\sigma_{k-1} (\varphi_0 + \psi_i) = k+p_{k-1} (n+i) = p_{k-1} (i)  &\mbox{   for $i<b$} \\
\sigma_{k-1} (\varphi_1 + \psi_i) = p_{k-1} (n+i)  &\mbox{   for $i<b$} \\
\sigma_{k-1} (\psi_i) = p_{k-1} (i)  &\mbox{   for $i \geq b$.}
\end{align*}
Since $p_{k-1} (i) > p_{k-1} (i+1)$ for all $i$, we see that all of these slopes are distinct.
\end{proof}

\section{A lifting result for maps to toric varieties}
\label{Sec:Def}

In this section, we prove our main tropical lifting result, Theorem~\ref{thm: lifting}, for tropical stable maps from chains of cycles. Once the lifting result has been established, an analysis of piecewise linear functions in linear systems on special chains of cycles, in the spirit of~\cite{JP14,JP16}, gives us the desired lifting for divisors. This is carried out in the next section.

\subsection{Strategy} As the proof will pass through several reductions, we begin with an outline of the general strategy. Given a tropical map $\Gamma\to \RR^r$, we choose a complete fan $\Delta$ supported on $\RR^r$.  After a subdivision of $\Gamma$, we interpret the map $\Gamma\to \Delta$ of polyhedral complexes as encoding logarithmic, or ``zeroth'' order, deformation data of a marked nodal curve $C_0$ mapping to the toric variety $Z$ defined by $\Delta$. As we have seen in Section~\ref{Sec:Deform}, $\Gamma\to \Delta$ determines a single point in a moduli cone $T_\Theta$ parametrizing tropical stable maps to $\Delta$ with the given combinatorial type. Recall that on the algebraic side, there is a moduli space $\mathscr L$ of deformations of the map $C_0\to Z$, and the analytic formal fiber $\mathscr L^\beth$ admits a continuous tropicalization map $\trop: \mathscr L^\beth\to T_\Theta$ by work of the second author~\cite{R16a}.

We must determine when the chosen moduli point $\Gamma\to \Delta$ lies in the image of this tropicalization map. By the general machinery of perfect obstruction theories in the logarithmic setting, the deformation space $\mathscr L$ is cut out of the formal toric variety defined by the cone $T_\Theta$, by equations corresponding to the obstructions\footnote{We direct the reader to the MathOverflow post of Jonathan Wise for an illuminating discussion~\cite{WisOverflow}.}. The strategy is to place sufficient hypotheses on $\Gamma\to \RR^r$ to detect the tropicalizations of these equations, thus characterizing the image of the tropicalization map above. For the chains of cycles, as previously discussed, obstructions appear when deforming characters that are orthogonal to the cycles. Our hypotheses ensure that these obstructions can be analyzed ``cycle-by cycle'' and ``character-by-character'', see Assumption~\ref{assumption}.

Throughout, we use the birational invariance properties of logarithmic stable maps~\cite{AW}, which ensure that we may perform toric modifications to the toric variety at various points in our analysis without changing the realizable locus in $T_\Theta$. Together with the combinatorial interpretation for the obstruction space given by Cheung, Fantini, Park, and Ulirsch~\cite[Section 4]{CFPU}, this allows us to present the space of maps to $Z$ as a logarithmically smooth fibration over a space of maps to a product of projective lines. In turn, this can be studied as a fiber product of spaces of maps to $\PP^1$. It is a special property of the chain of cycles, and in fact of our specific geometry, that the obstructions to tropical lifting all arise from lifting characters.

The spaces of maps to $\PP^1$ can be analyzed using Berkovich analytic techniques. Our hypotheses will be such that adjacent parts of the cycle can be lifted independently of each other via Speyer's results~\cite{Sp07}, and these analytic lifts are pasted together using the gluing procedure of Amini, Baker, Brugall\'e, and Rabinoff~\cite{ABBR}. Via lifting theorems for tropical intersections, the analysis on each $\PP^1$ characterizes a set of tropical curves in $T_\Theta$ for which all obstructions vanish, and we obtain the requisite realizability theorem.

\subsection{Statement of the lifting criterion} Let $T$ be an algebraic torus of dimension $r$ with character lattice $M$ and cocharacter lattice $N$. We write $N_\RR$ for $\Hom(M,\RR)$. Our primary interest will be in curves of the following form.

\begin{assumption}\label{assumption}
Let $[\Gamma\to \Delta]$ be a trivalent, maximally degenerate combinatorial type for a tropical stable map.
\begin{enumerate}[(A)]
\item The image of $\Gamma$ in $\mathcal M_g^{\trop}$ is a chain of cycles, see Figure~\ref{Fig:TheGraph}.
\item The edge directions of each cycle of $\Gamma$ spans a subspace in $N_\RR$ of codimension at most $1$.
\item Consecutive cycles are transverse, i.e. the edge directions of any two consecutive cycles span $N_\RR$.
\end{enumerate}
\end{assumption}

Note that every tropical curve $\Gamma$ of genus $g\geq 1$ has a well-defined stable metric graph $\Gamma^{\mathrm{st}}$, given by forgetting all the ends and stabilizing. Given a point $v\in\Gamma$, we will refer to its image in $\Gamma^{\mathrm{st}}$ as the \textit{stabilization} of $v$.

\begin{definition}\label{def: naive-well-spacedness}
Assume $\Psi : \Gamma\to \Delta$ is a tropical stable map such that the image of $\Gamma$ in $\mathcal M_g^{\trop}$ is a chain of $g$ cycles. Let $H\subset N_\RR$ be the hyperplane containing a cycle of $\Gamma$. Let $\Gamma_H^i$ be the connected component of $\Psi^{-1}(H\cap f(\Gamma))$ containing the $i^{\mathrm{th}}$ cycle $\gamma_i$ of $\Gamma$. Denote the $1$-valent vertices of $\Gamma_H^i$ by $v^i_1,\ldots,v^i_k$ and by $\ell_j^i$ the distance from $v_j^i$ to $\gamma_i$.
\begin{enumerate}[(WS1)]
\item The map $\Psi$ is \textbf{well-spaced at $\gamma_i$ with respect to $H$} if the the minimum of the multiset of distances $\{\!\{\ell^i_1,\ldots, \ell^i_k\}\!\}$ occurs at least twice, and for at least one of the points $v_j^i$ where this minimum is achieved, the stabilization of $v_j^i$ lies on $\gamma_i$.
\item The map $\Psi$ is said to be \textbf{well-spaced with respect to $H$} if it is well-spaced at all cycles $\gamma_i$.

\item The map $\Psi$ is said to be \textbf{naively well-spaced} if it is well-spaced with respect to all hyperplanes.
\end{enumerate}
\end{definition}

Well-spacedness in genus $1$ is a sufficient condition for realizability by work of Speyer~\cite{Sp07}. We refer to the above form of well-spacedness for the chain of cycles as ``naive'' because to expect realizability of such curves assumes that there are no further obstructions than the ones contributed by each cycle individually. Our main result of this section is that, if the map satisfies Assumption~\ref{assumption}, then naive well-spacedness is a sufficient condition for lifting.

\begin{theorem}\label{liftingtheorem}
Let  $\Psi: \Gamma\to \Delta$ be a tropical prestable map satisfying the conditions of Assumption~\ref{assumption}. If $\Psi$ is naively well-spaced, then there exists a smooth genus $g$ algebraic curve $C$ over a valued field $K$ extending $\CC$, and a map
\[
F: C\to Z
\]
such that the tropicalization of $F$ is $\Psi$.
\end{theorem}

Note that Theorem~\ref{thm: lifting} follows immediately from this result, by passing to the open subscheme of $C$ that maps to the dense torus of $Z$.

\subsection{Maps to the projective line} In this subsection, we establish the lifting theorem for the special case of maps to $\PP^1$. The result is proved via a reduction to the genus $1$ case, using semistable vertex decompositions for morphisms of analytic curves. The reader is encouraged to think of this as a non-archimedean analytic analogue of building a map to $\PP^1$ by cutting source and target along simple closed curves, building a collection of maps, and then gluing them together.

\begin{theorem}\label{thm: mapstoP1}
Let $\varphi: \Gamma\to \RR$ be a tropical map satisfying Assumption~\ref{assumption}. There exists a smooth proper curve $C$ of genus $g$ over a non-archimedean extension of $\CC$ and a map $f: C\to \PP^1$ with $f^{\trop} = \varphi$ if and only if $\varphi$ is naively well-spaced.
\end{theorem}

\begin{proof}
By the transversality part of Assumption~\ref{assumption}, if $\varphi$ contracts the $j^{\mathrm{th}}$ cycle, the $(j-1)^{\mathrm{st}}$ and $(j+1)^{\mathrm{st}}$ cycles remain uncontracted. For efficiency of bookkeeping, assume that $j$ is neither $1$ nor $g$. In these edge cases, we leave the mild adjustments to the reader. Choose an open interval $I$ in $\RR$  such that each connected component of $\varphi^{-1}(I)$ contains at most one cycle. If this cycle is labeled $\gamma_j$, we can ensure that $\varphi^{-1}(I)$ contains a nonempty open set in the cycles $\gamma_{j-1}$ and $\gamma_{j+1}$. See Figure~\ref{fig: surgery} below.
\begin{figure}[h!]
\begin{tikzpicture}

\draw (1.5,0) circle (0.5);
\draw (2,0)--(2.5,0);

\draw (3.5,0)--(4,0);
\draw (2.5,0) to [bend right] (3.5,0);
\draw (2.5,0) to [bend left] (3.5,0);

\draw (4.5,0) circle (0.5);

\draw[<->] (6,1.5)--(6,-1.5);
\draw [ball color=black] (6,0) circle (0.55mm);

\node at (5.5,0) {$\to$};

\begin{scope}[shift={(1,0)}]
\node at (5.75,0) {$\rightsquigarrow$};

\draw (6.5,0.5)--(7,0);
\draw (6.5,-0.5)--(7,0);
\draw (7,0)--(7.5,0);

\draw (7.5,0) to [bend right] (8.5,0);
\draw (7.5,0) to [bend left] (8.5,0);

\draw (9,0)--(8.5,0);
\draw (9,0)--(9.5,0.5);
\draw (9,0)--(9.5,-0.5);


\draw (11,-0.35)--(11,0.35);
\draw[densely dotted] (11,-0.35)--(11,-0.5);
\draw[densely dotted] (11,0.35)--(11,0.5);
\draw [ball color=black] (11,0) circle (0.55mm);
\end{scope}

\end{tikzpicture}
\caption{The figure on the left shows a superabundant tropical curve. The map is given by projection onto the vertical axis. The center cycle has been made thin to indicate that it is contracted. The figure on the right shows a neighborhood of $0$ in the target and a connected component of its preimage. We suppress the trees attached to the cycles for clarity. }\label{fig: surgery}
\end{figure}

Repeat this procedure as follows. Decompose $\RR$ into open intervals $I_0\cup I_1\cup\cdots\cup I_s$, where $I_0$ and $I_s$ are unbounded. By taking preimages via $\varphi$, we obtain a corresponding decomposition of $\Gamma$ into possibly disconnected subgraphs $\Gamma_i$ and maps $\Gamma_i\to I_i$. We may choose this decomposition such that each connected component of $\Gamma_i$ contains no contracted cycles, or contains a unique cycle that is contracted by the map to $I_i$. See Figure~\ref{fig: surgery}.

Fix $\Gamma_i\to I_i$. We claim that there exists a (non-proper) Berkovich analytic curve $E_i$ over a non-archimedean field $K$ extending $\CC$, with a map to a Berkovich analytic (generalized) annulus $A_i$, in the sense of~\cite{BPR13}, whose tropicalization is $\Gamma_i\to I_i$. To see this, extend the bounded $1$-valent vertices of $\Gamma_i$ and $I_i$ to infinite edges. This produces a tropical map $\widehat \Gamma_i \to \RR$. By Speyer's well-spacedness condition~\cite[Theorem 3.4]{Sp07}, there exists a punctured elliptic curve $\widehat E_i$ and an invertible function $\widehat E_i\to \mathbb G_m$ whose tropicalization is $\widehat \Gamma_i \to \RR$. That is, we have a commutative diagram
\[
\xymatrix{
\widehat E_i^{\an} \ar[d] \ar[r]^{\trop} & \widehat \Gamma_i \ar[d] \\
{\mathbb G}_m^{\mathrm{an}} \ar[r]_{\trop} & \RR  ,}
\]
where the horizontal arrows are projections to skeletons. The preimage of $I_i\subset \RR$ under the tropicalization map yields an annulus $A_i$, while the preimage of $\Gamma_i\subset\widehat \Gamma_i$ in $\widehat E_i^{\an}$ yields the desired analytic curve $E_i$.

Repeat this process for all indices $i$ corresponding to genus $1$ graphs with a contracted cycle. For the remaining $i$, the maps are lifted in analogous fashion. Note that there are no obstructions to lifting these maps~\cite{ABBR,CMR14a}. This process produces a collection of maps from analytic curves to star shaped neighborhoods in $\mathbb G_m^{\an}$, in the sense of~\cite[Definition 6.2]{ABBR}. These local liftings can be pasted together  along annuli $\trop^{-1}(I_j\cap I_{j+1})$ via the procedure of~\cite[Section 7.2]{ABBR}. After compactifying, this produces a map of proper analytic curves $C^{\an}\to \PP^1_{\an}$ whose tropicalization is $\Gamma\to \RR$. By non-archimedean GAGA, this map algebraizes, and we conclude that $[\Gamma\to \RR]$ is realizable provided it is well-spaced. Since $\Gamma$ is assumed to be trivalent, the converse follows from the necessity of well-spacedness in the genus $1$ case~\cite[Theorem 6.9]{BPR16}.
\end{proof}

\subsection{Maps to toric targets}  To extend the result from target $\PP^1$ to a general toric target $Z$, we express the locus of curves admitting the necessary map to $Z$ as an intersection of loci of curves admitting maps to $\PP^1$ and then tropicalize this intersection. By a well-known tropical lifting theorem of Osserman and Payne, the tropicalization of the intersection can be computed as the intersection of the tropicalizations. We use the following fact throughout, which can be proved by hand, or alternatively follows immediately from~\cite[Section 4]{AW}.

\begin{lemma}\label{lem: birational-invariance}
Let $\Delta$ and $\Delta'$ be complete fans in $N_\RR$ and let $[\varphi_1: \Gamma\to \Delta]$ and $[\varphi_2: \Gamma'\to \Delta']$ be tropical prestable maps, whose underlying maps to $N_\RR$ coincide. Then $\varphi_1$ is realizable if and only if $\varphi_2$ is realizable.
\end{lemma}


We specialize to the chain of cycles geometry. Let $\Theta = [\Gamma\to \Delta]$ be the combinatorial type of a tropical stable map satisfying Assumption~\ref{assumption}, and let $T_\Theta$ be the associated moduli cone. Let $\chi_j\in M$ be the primitive character orthogonal the cycle $\gamma_j$, noting that some of these $\chi_j$ may coincide or be $0$ if the cycle spans $N_\RR$.  Note that by the assumption, there is at most one such nonzero character for each cycle. Since we may replace $Z$ with a modification, we assume that the characters induce a morphism
\[
(\chi_1,\ldots,\chi_g): Z\to (\PP^1)^g.
\]
Interpreting a character as a map $N_\RR\to \RR$, this also induces a map $\Delta\to \RR^g$. By composition, this produces from $\Theta$ a new combinatorial type $\Upsilon = [\Gamma\to \RR^g]$. Here $\RR^g$ is given the structure of the fan of $(\PP^1)^g$.

\begin{proposition}
\label{Prop:NWSisRealized}
Let $\varphi: \Gamma\to \RR^g$ be a tropical stable map of type $\Upsilon$, in the notation above. If $\varphi$ is naively well-spaced, then $\varphi$ is realizable.
\end{proposition}

\begin{proof}
Fix a minimal logarithmic stable map $[C_0\to (\PP^1)^g]$ whose combinatorial type is $\Upsilon$. Let $\mathscr L(\mathscr A_{\PP^1}^g,\chi_j)$ be the fiber product
\[
\xymatrix{\mathscr L(\mathscr A_{\PP^1}^g,\chi_j) \ar[d] \ar[r] & \mathscr L(\mathscr A_{\PP^1}^g)  \ar[d] \\
\mathscr L(\PP^1) \ar[r]& \mathscr L(\mathscr A_{\PP^1})  ,}
\]
where the right vertical is induced by $\chi_j$, and the bottom arrow is the induced map to the Artin fan. The moduli space $\mathscr L(\mathscr A_{\PP^1}^g,\chi_j)$ parametrizes families of maps to $\mathscr A_{\PP^1}^g$ together with a compatible lift of the $j^{\textrm{th}}$ projection to $\PP^1$.  To simultaneously lift all the projections, observe that we have a natural morphism
\[
\mathscr L(\mathscr A_{\PP^1}^g,\chi_1)\times_{\mathscr L(\mathscr A_{\PP^1}^g)} \cdots \times_{\mathscr L(\mathscr A_{\PP^1}^g)} \mathscr L(\mathscr A_{\PP^1}^g,\chi_g)\to \mathscr L(\mathscr A_{\PP^1}^g).
\]
The image of the formal fiber of this fiber product under the map
\[
\trop: \mathscr L(\mathscr A_{\PP^1}^g)^{\beth}\to T_{\Upsilon}
\]
is the tropicalization of the intersection of the images of all $\mathscr L(\mathscr A_{\PP^1}^g,\chi_j)$ in $\mathscr L(\mathscr A_{\PP^1}^g)$. By the discussion in Section~\ref{sec: realizability}, the tropicalization of this intersection is the locus of realizable curves in $T_{\Upsilon}$. On the other hand, by Theorem~\ref{thm: mapstoP1}, the image of
\[
\trop: \mathscr L(\mathscr A_{\PP^1}^g,\chi_j)^{\beth}\to\mathscr L(\mathscr A_{\PP^1}^g)^{\beth}\to T_{\Upsilon}
\]
is the locus of maps in $T_{\Upsilon}$ such that all cycles orthogonal to the character $\chi_j$ are naively well-spaced. Assume that $\chi_j$ is nonzero. By the Riemann--Hurwitz formula for maps to $\PP^1$, we see that
 the codimension of the locus of maps in $\mathscr L(\mathscr A_{\PP^1}^g)$ where the $j^{\mathrm{th}}$ projection lifts is equal to the number of cycles in $\Upsilon$.

Consider nonzero characters $\chi_{j_1}$ and $\chi_{j_2}$, and examine the codimension of the intersection
 \[
\trop(\mathscr L(\mathscr A_{\PP^1}^g,\chi_{j_1})) \cap \trop(\mathscr L(\mathscr A_{\PP^1}^g,\chi_{j_2})).
 \]
 The expected codimension of this intersection is equal to the number of cycles that are contracted by either $\chi_{j_1}$ or $\chi_{j_2}$.
If the codimension of the intersection is equal to the expected codimension, the tropicalizations of all $\mathscr L(\mathscr A_{\PP^1}^g,\chi_j)$ intersect properly. This is guaranteed by the condition (WS1) in the definition of naive well-spacedness, and in particular the condition that one of the minimum distances measured from a given loop occurs at a vertex that stabilizes to the loop in question. We apply~\cite[Theorem 1.1]{OP} to conclude that tropicalization commutes with intersection and it follows that the realizable locus coincides with the naively well-spaced locus, as desired.
\end{proof}

Let $\Theta$ continue to be the combinatorial type chosen at the beginning of this section, and $\Upsilon = [\Gamma\to \RR^g]$ the map induced by projections orthogonal to the characters. From the above proposition, we may lift each of the projections induced by the characters $\chi_1,\ldots, \chi_g$ to the same curve. However, these characters need not span a rank $g$ sublattice of $M$. The following proposition guarantees that the relations between these characters do not provide any further obstructions to lifting. The sublattice generated by $\chi_1,\ldots,\chi_g$ determines a quotient $N_\RR\to L_\RR$. Choose a complete fan structure $\Delta_L$ on $L_\RR$ such that $\Delta\to \Delta_L$ is a morphism of fans, inducing a dominant equivariant map $Z\to Z_L$ of toric varieties.

\begin{proposition}\label{lifting-ob-directions}
Let $\Gamma\to\Delta$ be a tropical stable map of type $\Theta$ that is naively well-spaced. Then the induced map $\Gamma\to \Delta_L$ is realizable.
\end{proposition}

\begin{proof}
The characters $\chi_1,\ldots,\chi_g$ determine a short exact sequence of cocharacter lattices
\[
0\to L_\RR\to \RR^g\to H_\RR\to 0,
\]
where the dual lattice of $L$ is generated by $\chi_1,\ldots, \chi_g$ in $M$. By birational invariance, we may assume that the surjection above is induced by a dominant equivariant toric morphism
\[
(\PP^1)^g\to Z_H,
\]
where $Z_H$ is an equivariant compactification of the quotient torus of $\mathbb G_m^g$ determined by $H$. Choose minimal logarithmic map $[f_0: C_0\to (\PP^1)^g]$ whose combinatorial type is $\Upsilon$. The morphism above induces a map on local moduli spaces (i.e. of logarithmic deformation spaces)
\[
\mathscr L((\PP^1)^g)\to \mathscr L(Z_H).
\]
Let $\mathscr C\to (\PP^1)^g$ be a flat family of minimal logarithmic stable maps over an irreducible base scheme $S$, containing $[C_0\to (\PP^1)^g]$. Every fiber in the induced family of maps $\mathscr C\to (\PP^1)^g\to Z_H$ contracts the source curve to a point. To see this, first observe that, by construction, the combinatorial type induced by composing $\Upsilon$ with $\RR^g\to H_\RR$ contracts the source graph. Thus, the contact orders of $[C_0\to Z_H]$ with the toric boundary are all trivial. Since $Z_H$ is a proper toric variety, $C_0$ must be contracted. Contact orders are locally constant in flat families, so the same is true for every map parametrized by $S$. It follows that we have a morphism
\[
\alpha: \mathscr L((\PP^1)^g)\to Z_H,
\]
sending a family of maps to the family of points in $Z_H$ to which the source curves are contracted. The moduli space of maps $\mathscr L(Z_L)$ is the scheme theoretic fiber $\alpha^{-1}(\underline 1)$ of the identity in the dense torus of $Z_H$. By functoriality of tropicalization for logarithmic schemes, the tropicalization of $\alpha^{-1}(\underline 1)$ is the locus of tropical maps in $T_{\Upsilon}$ that lie in the fiber above zero of the induced map $T_{\Upsilon}\to \Delta_H$.  By Proposition~\ref{Prop:NWSisRealized}, this in turn coincides with the locus of maps $\Gamma\to \Delta_L$ that are naively well-spaced, and the result follows.
\end{proof}

We conclude the proof of the lifting theorem. It remains to show that the complementary characters to $\chi_1,\ldots,\chi_g$ can be lifted without further obstructions. Since the complementary characters are precisely those that are orthogonal to none of the cycles, this follows in similar fashion to the arguments presented in Section~\ref{Sec:GenericLift}. The formal argument is given below.

\subsubsection{Proof of Theorem~\ref{liftingtheorem}} Let $F: \Gamma\to \Delta$ be a naively well-spaced tropical stable map satisfying Assumption~\ref{assumption}. Let $[C_0\to Z]$ be a minimal logarithmic stable map with combinatorial type the same as that of $\Gamma\to \Delta$. As in the preceding discussion, let $L_\RR$ be the quotient of $N_\RR$ dual to the subgroup generated by $\chi_1,\ldots,\chi_g$. Choose a splitting $N \cong L \oplus L'$. By the birational invariance of Lemma~\ref{lem: birational-invariance}, we may assume that $\Delta$ is of product type $\Delta_L\times \Delta_{L'}$, for complete fans in $L_\RR$ and $L_{\RR}'$. Let $Z = Z_L\times Z_{L'}$ be the corresponding product of toric varieties. The projection $Z\to Z_{L}$ determines a morphism of Artin fans $\mathscr A_Z\to \mathscr A_{Z_{L}}$ and similarly with $L'$. 

Let $T(\Delta)$ be the moduli cone of maps whose type is that of $F$ and let $T(\Delta_L)$ and $T(\Delta_{L'})$ be the analogous moduli of maps determined by composing $F$ with each of the two projections above. Finally, let $T_\Gamma$ be the cone in the moduli of tropical prestable curves with combinatorial type that of $\Gamma$. We now show that a moduli point in $T(\Delta)$ associated to a naively well-spaced map is realizable, i.e. that such a point lies in the image of the tropicalization map
\[
\mathscr L(Z)^\beth\to \mathscr L(\mathscr A_Z)^\beth\to T(\Delta).
\]
Define $\mathscr L(\mathscr A_Z,L)$ by the fiber product
\[
\xymatrix{\mathscr L(\mathscr A_{Z},L) \ar[d] \ar[r] & \mathscr L(\mathscr A_{Z})  \ar[d] \\
\mathscr L(Z_{L}) \ar[r]& \mathscr L(\mathscr A_{Z_{L}})  ,}
\]
and analogously for $L'$. In the cone $T_\Gamma$, the image of $\mathscr L(\mathscr A_{Z},L)^\beth$ under tropicalization map
\[
\mathscr L(\mathscr A_{Z},L)^\beth\to T(\Delta)\to T_\Gamma
\]
 is the locus of tropical curves that admit a lift with a map to $Z_L$ with the combinatorial type determined by the projection. By Proposition~\ref{lifting-ob-directions}, this includes the locus of naively well-spaced curves. Since the type $T(\Delta_{L'})$ is non-superabundant, its image in $T_\Gamma$ coincides with the image of $\mathscr L(\mathscr A_{Z},L')^\beth$ under the tropicalization map
\[
\mathscr L(\mathscr A_{Z},L')^\beth\to T(\Delta)\to T_\Gamma
\]
The intersection of the images of $T(\Delta_L)$ and $T(\Delta_{L'})$ in $T_\Gamma$ has the expected codimension. By the product decomposition of $Z$, a map from a family of curves to $Z$ is precisely given by maps from this family to $Z_L$ and $Z_{L'}$. Now applying the tropical lifting result of~\cite[Theorem 1.1]{OP}, the locus of naively well-spaced maps lies in the image of the tropicalization map $\mathscr L(Z)^\beth\to T(\Delta)$. The result follows. \qed

\section{Lifting divisors on a $k$-gonal chain of cycles}
\label{Sec:SpecialLift}

We now return to the chain of cycles $\Gamma$ with torsion profile as specified in Section~\ref{Sec:ScrollarTableaux}.
Throughout this section, we consider scrollar tableaux $t$ with the additional property that they have no \emph{vertical steps}.  That is,
\[
t(x,y+1) \neq t(x,y)+1 \mbox{   for all } x,y .
\]
Note that such tableaux exist as long as $n>1$.  When $n=1$, the scroll $\mathbb{S}(1,0)$ is just $\PP^1$ and the argument below will work without this assumption, see Remark~\ref{rem: vertical-remark}.  We let $D \in \mathbf{T}(t)$ be a sufficiently general divisor.  We let $m = \lfloor \frac{r+1}{n} \rfloor$ and write $\vec{p}$ for the lingering lattice path corresponding to the tableau $t(-m+1)$.

As in Section~\ref{Sec:GenericLift}, we modify the graph $\Gamma$ as follows.  For every point in the support of $\ddiv \varphi_i$ or $\ddiv \psi_i$, we attach a tree, trivalent away from the attaching point, based at this point.  The trees are based at the points $\langle j \rangle_j$ and $\langle -j-1 \rangle_{g+1-j}$ for $j<k$, and $\langle \eta_j - p_{j-1} (i) \rangle_j$ for every $i$ and every $j$ that does not appear in the $i^{th}$ column of the tableau $t$.  The number of leaves of each tree is equal to the number of divisors that contain that point in their support.  Note that, since $D(-m)$ is vertex avoiding, none of these trees are based at $v_j$ or $w_j$ for any $j$.  Attach infinite rays based at $v_1$ and $w_g$. In a mild abuse of notation, we continue to denote the modified graph by $\Gamma$.

We now construct a map $\Psi : \Gamma \to \mathbb{R}^n$.  If $b=0$, the map is defined by
\[
\Psi(p) = \left( \varphi_0 (p), \psi_0 (p) , \ldots , \psi_{a-2} (p) \right) .
\]
Otherwise, the map is given by $a+b$ piecewise linear functions
\[
\Psi(p) = \left( \varphi_0 (p), \psi_0 (p) , \ldots , \psi_{b-2} (p), \psi_b (p), \ldots , \psi_{n-1} (p) \right) .
\]
If $b=1$, then the list of functions $\psi_0 , \ldots , \psi_{b-2}$ should be taken as the empty collection.  Note that, in the case $a=1$, $b=0$, we have $\Psi = \varphi_0$.  The function above has been defined on the chain of cycles itself. By enforcing the balancing condition, if we fix the unbounded leaf edge directions, this piecewise linear map becomes uniquely determined on the trees attached to each cycle. By construction the unbounded edges of $\Gamma$ are parallel to the rays of the fan of $\mathbb S(a,b)$. We now consider the span of the edge directions in each cycle under this piecewise linear map.

\begin{proposition}
\label{Prop:SpecialSpan}
Let $\gamma_j \subset \Gamma$ be the $j^{\mathrm{th}}$ cycle of $\Gamma$.
\begin{enumerate}
\item  If $k \leq j \leq g-k+1$, then $\Psi(\gamma_j)$ spans the vector space $\mathbb{R}^n$.
\item  Otherwise, $\Psi(\gamma_j)$ spans a linear space of codimension at most 1.
\item  The span of any two consecutive cycles is the whole vector space $\mathbb{R}^n$.
\end{enumerate}
\end{proposition}

\begin{proof}
By the balancing condition, the span of each cycle $\Psi(\gamma_j)$ is equal to that of the roots of the trees based at that cycle, plus that of the two bridges emanating from that cycle.  To match the notation of Section~\ref{Sec:Scrolls}, we let $u_1$ denote first coordinate vector in $\mathbb{R}^n$, and $e_1, \ldots, e_{n-1}$ the remaining coordinate vectors.  We will write $e_0 = -\sum_{i=1}^{n-1} e_i$ and $u_0 = -u_1 -\sum_{i=b}^{n-1} e_i$.

First consider cycles $\gamma_j$ where $k \leq j \leq g-k+1$.  Since $p_j (i) > p_j (i+1)$ for all $i$ and $p_j (b-1) = k+ p_j(n+b-1)$, the values $p_j (i)$ are distinct mod $k$ for $b \leq i \leq n+b-1$.  It follows that there are at least $n-1$ infinite rays attached to $\gamma_j$ that are all based at distinct points.  These infinite rays are translates of $n-1$ of the vectors $e_0, \ldots , e_{n-1}$.  Since any $n-1$ of these vectors are linearly independent, we see that the span of $\Psi(\gamma_j)$ contains the span of $e_1, \ldots , e_{n-1}$.  The two bridges emanating from the cycle are translates of the vectors $(k, p_{j-1}(n), \ldots , p_{j-1}(n+b-2), p_{j-1}(b) , \ldots, p_{j-1} (n-1))$ and $(k, p_j(n), \ldots , p_j(n+b-2), p_j(b) , \ldots, p_j (n-1))$.  Since the first coordinate is $k \neq 0$, the span of $\Psi(\gamma_j)$ is the whole space $\mathbb{R}^n$.

Next, consider the cycles $\gamma_j$ where $j \geq g-k+2$.  If none of the divisors $\ddiv \psi_i$ contain the point $\langle j-(g+2) \rangle_j$ in their support, then there are at least $n$ infinite rays attached to $\gamma_j$ that are all based at distinct points.  One of these infinite rays is a translate of the vector $u_1$, and the remaining ones are translates of $e_0 , \ldots , e_{n-1}$.  Since any $n$ of these vectors is linearly independent, the cycle spans the whole space.

On the other hand, at most one of the divisors $\ddiv \psi_i$ may contain the point $\langle j-(g+2) \rangle_j$ in its support.  In this case, there are at least $n-2$ infinite rays attached to $\gamma_j$ that are all based at distinct points, and one tree based at $\langle j-(g+2) \rangle_j$ with two leaves.  The root of the tree is a translate of the vector $e_y + u_1$ for some $y$.  The remaining rays are translates of $e_0 , \ldots , \widehat{e_y}, \ldots , e_{n-1}$.  Since these vectors are linearly independent, the span of $\Psi(\gamma_j)$ has codimension at most one.  In the case where it has codimension exactly one, we have $p_j - p_{j-1} = e_x -u_1$ for some $x$, and the span does not contains a translate of the vector $e_x$, or of the vector $e_x + u_1$.

To see that any two consecutive cycles span the whole space, consider the span of $\Psi(\gamma_{j-1})$.  By our assumption that there are no vertical steps, we have $p_{j-1}-p_{j-2} \neq e_x - u_1$.  Thus, the span of $\Psi(\gamma_{j-1})$ contains either a translate of $e_x$ or of $e_x + u_1$.  In either case, the span contains a vector that was not contained in the span of $\Psi(\gamma_j)$.  It follows that any two consecutive cycles are transverse.

The statement for cycles $\gamma_j$ for $j < k$ follow by a similar argument.  At most one of the divisors $\ddiv \psi_i$ may contain the point $\langle j \rangle_j$ in its support.  There are therefore at least $n-2$ infinite rays attached to $\gamma_j$ that are all based at distinct points, and if there are exactly $n-2$, then there is one tree based at $\langle j \rangle_j$ with two leaves.  The root of the tree is a translate of the vector $e_y + u_0$ for some $y$.  The remaining rays are translates of $e_0 , \ldots , \widehat{e_y}, \ldots , e_{n-1}$.  Since these vectors are linearly independent, we see that the span of $\Psi(\gamma_j)$ has codimension at most one.  In the case where it has codimension exactly one, we have $p_j - p_{j-1} = e_x -u_0$ for some $x$, and the span does not contains a translate of the vector $e_x$, or of the vector $e_x + u_0$.  The fact that any two consecutive cycles span the whole space follows by the same argument as above.
\end{proof}

\begin{remark}\label{rem: vertical-remark}
In the above proof, the condition that the tableau $t$ has no vertical steps is only used when $j < k$ and one of the divisors $\ddiv \psi_i$ contains the point $\langle j \rangle_j$ in its support, or when $j \geq g-k+2$ and one of the divisors $\ddiv \psi_i$ contains the point $\langle j-(g+2) \rangle_j$ in its support.  This hypothesis on the tableau can therefore be weakened without altering the proof of Proposition~\ref{Prop:SpecialSpan}.
\end{remark}

We now show that we may choose the edge lengths of $\Gamma$ so that the map $\Psi$ is naively well-spaced.  Moreover, this can be done in such a way that the dimension of the space of metric graphs with this combinatorial type and torsion profile is $2g-5+2k$. Assume that if $i \leq k-1$, then $n_i \gg n_{i+1}$.  This guarantees that, if $\gamma_i$ is contained in a hyperplane, then the closest point to $\gamma_i$ at which $\Gamma$ leaves the hyperplane is on $\gamma_{i+1}$.  In this case, there is a single tree based at $\gamma_i$ with two leaves, and $\Gamma$ also leaves the hyperplane at the branch point of this tree.  We are free to choose the distance from this branch point to $\gamma_i$. Define it to be the same distance from $\gamma_i$ as the closest point to $\gamma_i$ at which $\Gamma$ leaves this hyperplane.  In complementary fashion, if $i \geq g-k+2$, we assume that $n_i \gg n_{i-1}$, and define edge lengths in the analogous way.

\begin{proposition}
\label{Prop:SpecialLift}
Let $t$ be a scrollar tableau with no vertical steps, let $D \in \mathbf{T}(t)$ be a sufficiently general divisor class, and $\Psi$ the corresponding map.  Then there exists a curve $C$ of genus $g$ and gonality $k$ over a valued field $K$ and a divisor $\cD \in W^r_d (C)$, specializing to $D$ on $\Gamma$.
\end{proposition}

\begin{proof}
By Proposition~\ref{Prop:SpecialSpan}, the map $\Psi$ satisfies Assumption~\ref{assumption}, and thus by Theorem~\ref{liftingtheorem} there exists a smooth curve $C$ of genus $g$ over a valued field $K$ and a map $F: C \to \mathbb{S}(a,b)$ such that the tropicalization of $F$ is $\Psi$.  In the case where $b>0$ map $F$ is given coordinatewise by
\[
F(p) = \left( s_0 (p), f_0(p), \ldots , f_{b-2} (p), f_b (p) , \ldots , f_{n-1} (p) \right) ,
\]
where $\trop (s_0) = \varphi_0$ and $\trop (f_i) = \psi_i$. The case where $b  = 0$ is similar, and we leave the details of this case to the reader.  By Lemma~\ref{Lem:Independent}, the functions
\[
\varphi_0 + \psi_0 , \ldots , \varphi_0 + \psi_{b-1} , \varphi_1 + \psi_0 , \ldots , \varphi_1 + \psi_{b-1} , \psi_b , \ldots , \psi_{n-1}
\]
are tropically independent. Thus, the functions
\[
s_0f_0 , \ldots , s_0f_{b-1} , s_1 f_0 , \ldots , s_1 f_{b-1} , f_b , \ldots , f_{n-1}
\]
are linearly independent, and the map $F$ is nondegenerate.  We let $m = \lfloor \frac{r+1}{n} \rfloor$ and $\cD = F^\star (\cO_{\mathbb{S}(a,b)} (1) \otimes \cO_{\PP^1} (m))$.

By Proposition~\ref{Prop:MapToScroll}, $r(\cD - (m-1)g^1_k) \geq n+b-1$ and $r(\cD - mg^1_k) \geq b-1$.  To obtain equality, we note that by Corollary~\ref{Cor:Ranks} and Baker's specialization lemma, we have
\begin{enumerate}
\item  $r(\mathcal{D} - mg^1_k) \leq r(D(-m)) = b-1$, and
\item  $r(\mathcal{D} - (m+1)g^1_k) \leq r(D(-m-1)) = -1$.
\end{enumerate}
The result follows from Corollary~\ref{Cor:AllM}.
\end{proof}

We conclude the proof of Theorem~\ref{thm:mainthm} in analogous fashion to Theorem~\ref{thm: Kempf-Kleiman-Laksov}. By the result above, its proof is nearly identical to that of Theorem~\ref{thm: Kempf-Kleiman-Laksov}.

\begin{theorem}
Let $C$ be a general curve of genus $g$ and gonality $k$.  Then
\[
\dim W^r_d (C) = {\rho}_k (g,r,d).
\]
\end{theorem}

\begin{proof}
By \cite{Pfl16b}, we have $\dim W^r_d (C) \leq {\rho}_k (g,r,d)$, so it suffices to prove the reverse inequality. Let $\cM_g^k$ be the moduli space of curves of genus $g$ that admit a degree $k$ map to $\PP^1$.  A general member in this space is a curve of gonality $k$.  Let $\mathscr C_k$ be the universal curve, and let $\cW^r_d$ be the universal $W^r_d$ over $\cM_g^k$. As in the proof of Theorem~\ref{thm: Kempf-Kleiman-Laksov}, let $\widetilde \cW^r_d$ be locus in the symmetric $d^{\mathrm{th}}$ fibered power of $\mathscr C_k$ parametrizing divisors of degree $d$ and rank at least $r$. From the results of~\cite{CMR14a}, the tropicalization of $\cM_g^k$ coincides with the image of the tropical Hurwitz space of simply ramified degree $k$ covers of $\PP^1$ in $\Mbar_g^{\trop}$. As previously discussed, in the chain of cycles combinatorial type, this is seen to have dimension $2g-5+2k$.

We work in the Berkovich analytic domain of $k$-gonal curves whose skeleton is the chain of cycles with torsion profile as specified in Section~\ref{Sec:ScrollarTableaux}.  In the notation of Theorem~\ref{thm: Kempf-Kleiman-Laksov}, we have a continuous tropicalization map
\[
\trop: [\mathrm{Sym}^{d,\an}(\overline{\mathscr C},\Gamma)]\to \mathrm{Sym}^{d,\trop}(\Gamma).
\]

We can apply Proposition~\ref{Prop:SpecialLift} to conclude that the tropicalization of $\widetilde \cW^{r,\an}_d$  has dimension $2g-5+2k+{\rho}_k (g,r,d)+r$.  To see that $\cW^r_d$ has a component of dimension at least $2g-5+2k+\overline{\rho}_k (g,r,d)+r$, it suffices to show that, for a general pair $(C,D) \in \widetilde \cW^r_d$, the divisor $D$ has rank exactly $r$.  To see this, note that if $D$ has rank strictly greater than $r$, then for general points $p,q \in C$, $D+p-q$ has rank at least $r$.  If every divisor of rank at least $r$ has rank strictly greater than $r$, then by iterating this procedure we see that every divisor in $\Pic^d (C)$ has rank at least $r$, but this is impossible.  Therefore, $\cW^r_d$ has a component of dimension at least $2g-5+2k+{\rho}_k (g,r,d)$. By another application of Proposition~\ref{Prop:SpecialLift}, the image of this component in $\cM^{k,\text{trop}}_g$ has dimension $2g-5+2k$.  It follows that this component dominates $\cM^k_g$, and the fibers have dimension ${\rho}_k (g,r,d)$.
\end{proof}

\bibliographystyle{siam}
\bibliography{kGonalCurves}

\end{document}